\documentclass[10pt]{amsart}

\usepackage[english]{babel}
\usepackage[a4paper,top=2cm,bottom=2cm,left=3cm,right=3cm]{geometry}
\usepackage{float}

\usepackage{comment}
\usepackage[usenames,dvipsnames,table]{xcolor}
\newcommand{\newcomment}[4]{%
\newcounter{#2counter}
\expandafter\newcommand    \csname #1\endcsname[1]{%
\refstepcounter{#2counter}%
{\color{#4}(#3\arabic{#2counter})}\marginpar{\scriptsize\raggedright\textbf{\color{#4}(#2 \arabic{#2counter}):} ##1}%
}}

\newcommand\translate[2]{\overrightarrow{#1}^{#2}}
\usepackage{xurl}
\usepackage[shortlabels]{enumitem}
\usepackage{graphicx}
\usepackage{subcaption}
\newcomment{rafa}{Rafa}{R}{blue}
\newcomment{mik}{Mikael}{M}{red}

\usepackage{graphicx}
\usepackage{amsmath,amssymb,amsthm}
\usepackage{mathtools}
\usepackage[colorlinks=true, allcolors=blue]{hyperref}
\newcommand{\R}{\mathbf{R}}
\newcommand{\C}{\mathbf{C}}
\newcommand{\N}{\mathbf{N}}
\newcommand{\Z}{\mathbf{Z}}

\newcommand{\conv}{\operatorname{conv}}
\newcommand{\aff}{\operatorname{Aff}}
\newcommand{\ocone}[2]{\operatorname{Cone}^-(#1,#2)}

\newtheorem{theorem}{Theorem}[section]
\newtheorem{lemma}[theorem]{Lemma}
\newtheorem{proposition}[theorem]{Proposition}
\newtheorem{corollary}[theorem]{Corollary}
\newtheorem{definition}[theorem]{Definition}
\newtheorem{example}[theorem]{Example}
\newtheorem{remark}[theorem]{Remark}

\title{A construction of Kakeya Sets in Arbitrary Dimension}  
\author{Rafael J. Fernández-Delgado}\thanks{The first author
has received financial support from Spain's Ministerio de Ciencia, Innovación y Universidades
through a FPU Grant FPU24/00173.
}
\author{Mikael de la Salle}
\date{\today}

\begin{document}
\maketitle

\begin{abstract}  

 We construct Kakeya sets in arbitrary dimension $d\geq 2$, generalizing the classical Perron tree construction beyond dimension $2$. The Kakeya sets we construct have $\delta$-neighbourhood of volume at most $C|\log\delta|^{-(d-1)}$, which improves on previously known constructions, and which is conjecturally optimal. We further derive consequences for the $L^p$-boundedness of radial Fourier multipliers in terms of Besov spaces with logarithmic smoothness.
\end{abstract}

\section{Introduction}  
A Kakeya set in $\R^d$ is a subset containing a unit line segment in every direction. Originally a topic in recreational mathematics, Kakeya sets have evolved into a subject with profound connections to core mathematical fields, ranging from harmonic analysis, with Fefferman's multiplier problem for the ball \cite{RefWorks:1971multiplier}, to number theory \cite{RefWorks:bourgain2006remarks}. Recently the second-name author also realised their connection with operator algebras and the structure of the von Neumann algebras of high rank lattices \cite{delasalle2026kakeyaconjecturehighranklattice}. It is at first hand quite surprising that, in dimension at least $2$, there exist Kakeya sets of Lebesgue measure $0$. The classical construction of such Kakeya sets in $\R^2$, due to Besicovitch, simplified by Perron and further refined by Schoenberg \cite{Schoenberg1988OnTB} and Keich \cite{RefWorks:keich1999lp}, relies on the use of \emph{Perron trees} -- a family of line segments arranged in a hierarchical, tree-like structure. In this paper, we are interested in quantitative aspects of Kakeya sets. The most useful way to measure smallness of Kakeya sets $E$ will be described below, but for the moment let us decide to capture the smallness of $E$ as the behaviour, as $\delta$ goes to $0$, of the Lebesgue measure of the $\delta$-neighbourhood of $E$. In this perspective, the Perron tree construction is efficient, yielding Kakeya sets not only of Lebesgue measure $0$, but also optimal in the sense that the measure of their $\delta$-neighbourhood is $O(\frac{1}{|\log \delta|})$ as $\delta$ goes to $0$, thus being the best possible.

Extending this approach to higher dimensions has traditionally been limited to taking Cartesian products:
\begin{example} If $E \subset \R^2$ is an optimal Kakeya set, then $E \times [0,1]^{d-2}$ is a Kakeya set in $\R^d$, and the measure of its $\delta$-neighbourhood is $O(\frac{1}{|\log \delta|})$.
\end{example}
\begin{example}
  If $E \subset \R^2$ is an optimal Kakeya set and $d$ is even, then $E\times\dots \times E$ ($d/2$ times) is a Kakeya set in $\R^d$, and the measure of its $\delta$-neighbourhood is $O(\frac{1}{|\log \delta|^{d/2}})$.
\end{example}
While this method produces valid Kakeya sets that indeed have measure $0$, it fails to exploit the full geometric complexity of higher-dimensional spaces.

The Kakeya conjecture, which predicts that Kakeya sets cannot be too
small and must have full Hausdorff dimension,
remains a tantalizing open problem. Although not formally equivalent,
in first approximation it predicts that the measure of the
$\delta$-neighbourhood of a Kakeya set cannot decay as $O(\delta^\epsilon)$
for some $\epsilon>0$. While the conjecture has been long resolved
in two dimensions by Davies and Córdoba \cite{RefWorks:davies1971remarks,RefWorks:cordoba1977kakeya} using different methods, in higher dimensions, the problem becomes
significantly more challenging. The prevailing belief is that the
conjecture is \emph{harder} to prove in higher dimensions, due to the
increased combinatorial and geometric complexity.  And only recently
the three dimensional case has been settled in a celebrated
breakthrough by Wang and Zahl \cite{RefWorks:wang2025volume}.

Yet, paradoxically, as far as we are aware\footnote{No mention of known constructions where provided when the second-named authors asked this question on Mathoverflow in April 2025 \cite{MOQuestion}.}, prior to our work, there has not been any explicit construction of Kakeya sets in  $\R^d$  (for  $d \geq 3$) that demonstrates a quantitative advantage over the $2$-dimensional case. However, after a first version of this paper was made public on arXiv, Sam Craig informed of his work \cite{CraigJGA}, where he constructed non-trivial Kakeya sets in arbitrary dimensions.

In this paper, we present a construction of Kakeya sets in arbitrary dimension $d > 2$, which generalizes the Perron tree construction in dimension $2$. This construction does not seem to appear in the literature, although some forms of it has certainly been known to experts. As pointed out to us by Anthony Carbery when we communicated to him a preliminary version of our paper, such constructions were at least in the air in the 70's and 80's. We believe that, in particular in the light of the recent work \cite{delasalle2026kakeyaconjecturehighranklattice}, these constructions deserve to appear explicitly. 

\begin{theorem}\label{thm: main_kakeya} For every $d\in \N$, there exists a Kakeya set $E\subset \R^d$ and a positive real number $C$ such that for every $\delta \in (0,1)$, the volume of the $\delta$-neighbourhood of $E$ is less than  $\frac{C}{|\log\delta|^{d-1}}$.
\end{theorem}
In the vocabulary of Keich \cite{RefWorks:keich1999lp}, the exact Minkowski dimension of Kakeya sets in dimension $d$ is less than $\delta \mapsto \delta^d |\log \delta|^{d-1}$.

This result aligns with the Kakeya and other related conjectures, most likely being optimal. The reverse
Littlewood--Paley conjecture considered by Carbery \cite{carbery2015remarkreverselittlewoodpaleyrestriction}, for instance, implies that any Kakeya set has a $\delta$-neighbourhood volume greater than $\frac{C_d}{|\log\delta|^{d-1}}$ (see subsection~\ref{subsection:optimality}).

Theorem~\ref{thm: main_kakeya} should be compared to \cite[Proposition 3.1]{CraigJGA}. There, Craig constructs Kakeya sets in arbitrary dimension $d$,  satisfying the same conclusion as in Theorem~\ref{thm: main_kakeya}, but with volume estimates $C\Big(\frac{\log |\log \delta|}{|\log \delta|}\Big)^{d-1}$. So our construction improves on the $\log \log$ factor. More importantly, it allows us to infer additional geometric properties for $\delta$-tube configurations, that are essential for harmonic analysis applications that we discuss below.
\subsection{Application to Fourier analysis}

We now state a variant of Theorem~\ref{thm: main_kakeya}, better suited for applications to Fourier analysis. Here a $\delta$-tube $R$ is the image by an isometry of $[0,1]\times B_{\R^{d-1}}(0,\delta)$, and in that case, for $\alpha>0$, we denote by $\translate{R}{\alpha}$ the image of $[1+\alpha,2+\alpha]\times B_{\R^{d-1}}(0,\delta)$, that is, the translate of $R$ by $1+\alpha$ in the direction of its axis.
\begin{theorem}\label{thm:main_tubes} Let $d \geq 2$ and $\alpha\in (0,\infty)$. There is a constant $C_{d,\alpha}$ and, for every $\delta\in(0,1)$, a finite family $\delta$-tubes $R_1,\dots,R_N$ in $\R^d$ whose translates $\translate{R_i}{\alpha}$ are pairwise disjoint but
  \[ \frac{\big|\bigcup_{i=1}^N R_i\big|}{\sum_{i=1}^N |R_i|} \leq \frac {C_{d,\alpha}}{|\log \delta|^{d-1}}.\]
\end{theorem}

\begin{figure}[htbp]
    \centering
  \makebox[0pt][c]{\hspace{8cm}\begin{minipage}{0.48\textwidth}
    \includegraphics[width=\textwidth]{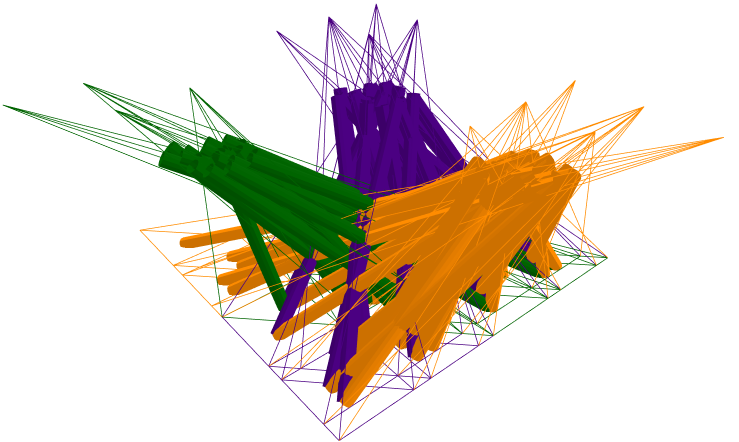}
  \end{minipage}}
  \hfill
   \makebox[0pt][c]{\hspace{-8cm}\begin{minipage}{0.48\textwidth}
    \includegraphics[width=\textwidth]{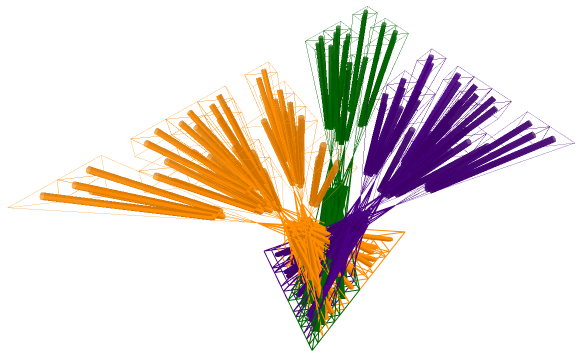}
  \end{minipage}}
    \caption{Illustration of the proof of Theorem \ref{thm:main_tubes} for $d=3$. The first image shows the tubes translated as a block within their respective pyramid during the $2$-iterations simplex Perron tree construction. The second includes the reflected pyramids across their respective apexes, which by construction contain the $\translate{R_i}{\alpha}$'s.}
\end{figure}
Using the results of the second-named author \cite{delasalle2026kakeyaconjecturehighranklattice}, we obtain the following consequence, which was our motivation for this work. The result is stated in terms of Besov spaces with logarithmic smoothness  (see subsection~\ref{subsection:besov} for definitions). 

It improves on \cite{delasalle2026kakeyaconjecturehighranklattice}, where the result was proved with $b=|\frac 1 p - \frac 1 2|$. The result is only relevant for $|\frac 1 p - \frac 1 2| \leq \frac{1}{2d}$, because otherwise it is known that $f$ is H\"older-continuous. 

\begin{corollary}\label{cor:radial_mult_sobolev}
  Let $1<p<\infty$ and $m:(0,\infty) \to \R$ a bounded measurable function such that $T$, the Fourier multiplier with symbol $m(|\cdot|)$ is bounded on $L_p(\R^d)$. Then $f:t \mapsto  m(\exp t)$ belongs to $B_{\infty,\infty}^{0,b} (\R)$ for $b=(d-1)\Big|\frac 1 p - \frac 1 2\Big|$.
\end{corollary}
For example, the Fourier multiplier with symbol $\xi\mapsto |\log \frac{(1-|\xi|)_+}{2}|^{-\beta}$ is not bounded on $L_p(\R^d)$ if $\beta < (d-1)|\frac 1 p - \frac 1 2|$ (see Example~\ref{ex:powerLogBesov}). For $d=2$, this has long been known to Carbery. It is natural to ask whether the converse is true. This is a logarithmic form of the Bochner-Riesz conjecture.

We would like to point out that struggling to improve the exponent of the $\log$ in Theorem~\ref{thm: main_kakeya} or Theorem~\ref{thm:main_tubes} can be meaningful. First any (unlikely) improvement on the exponent from $(d-1)$ to $c$ in Theorem~\ref{thm:main_tubes} would improve Corollary~\ref{cor:radial_mult_sobolev} and would in particular lead to new necessary conditions (namely $\beta \geq c|\frac 1 p - \frac 1 2|$) for the $L_p(\R^d)$-boundedness of the logarithmic Bochner-Riesz multipliers. Also, as proved by the second-name author \cite{delasalle2026kakeyaconjecturehighranklattice}, it would also imply that the non-commutative $L_p$ space of the von Neumann algebra of $\mathrm{SL}_{2d-1}(\Z)$ fails the operator space approximation property whenever $c |\frac 1 p - \frac 1 2|>1$. However, this would only be new if $c>2d$ because for $ 2d|\frac 1 p - \frac 1 2|>1$ the result has already been proven by other methods by Lafforgue, de Laat and the second-named author \cite{LafforguedlS,MR3781331}.

\subsection{Idea of the construction}

Both Theorem~\ref{thm: main_kakeya} and \ref{thm:main_tubes} rely on the same construction, which is an adaptation of the Perron tree construction (see Theorem~\ref{theo: perron tree}). In dimension $2$, the basic construction is the following: start with a triangle with a specified vertex $\mathbf{v}$. Cut it into two smaller triangles by dividing the edge opposite to the specified vertex, and translate the two triangles towards each other parallel to their bases. This basic step is subsequently repeated to each pair of neighbouring triangles formed at each iteration (see Figure \ref{fig: 2d perron}). 
\begin{figure}[htbp]
\centering
    \makebox[0pt][c]{\hspace{0.6cm}
    \includegraphics[scale=1.3]{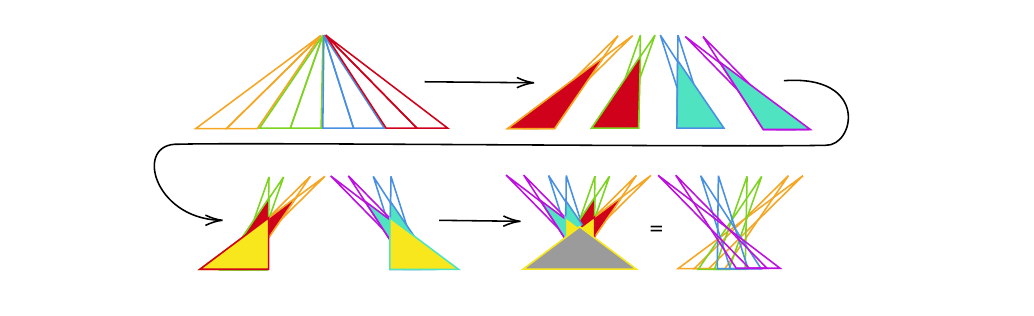}}
  
  \caption{Construction of a Perron tree via dividing a triangle into $ 2^{n}$ sub-triangles and recombining them with partial overlaps. Cases $n=3$ is illustrated. Triangles with the same colour are those to which the basic construction is applied in each step. The branches formed at each iteration are moved as a block along with the triangular cores from which they sprout.}
  \label{fig: 2d perron}
\end{figure}

The basic subdivision happens in the base. And the idea at the heart of the construction is the obvious fact that an interval can be divided into two smaller intervals of equal length, and that these intervals can be translated towards each other, in such a way that the total length of the interval formed by their union decreases. 

To generalize this process to higher dimension, we start with a $d$-dimensional pyramid with a distinguished vertex $\mathbf{v}$, the apex, and we have to study such partitions of the opposite face, which is a $(d-1)$-dimensional polytope, and how to rigidly move the pieces to form a strictly smaller copy of the original polytope. In the paper, we will be able to perform a Perron tree type construction starting from a rather general polytope and a partition of it, but for the introduction let us restrict ourselves to the case of a simplex.

A $(d-1)$-simplex $\Sigma$ is the convex hull of $d$ affinely independent points. Up to affine transformations, there is a unique $(d-1)$-simplex, but there are many ways of partitioning it into smaller simplices: the barycentric subdivision is probably the most common, but another more relevant for us is the Coxeter-Freudenthal-Kuhn subdivision. As we will explain in Section~\ref{sec: perron base}, if we start with a partition $(\Sigma_i)_{i \in I}$ of $\Sigma$ with the property that none of the pieces $\Sigma_i$ meets all the faces of $\Sigma$, it is possible to translate them into a strictly smaller copy of $\Sigma$, and we can analyse what these translations lead to at the level of the partition of the pyramid $T=\conv(\Sigma,\mathbf{v})$ into smaller pyramids $T_i=\conv(\Sigma_i,\mathbf{v})$. Choosing an identification of each of the pieces of the partition with $\Sigma$, we can further subdivide them to obtain a second-generation partition, and keep on going. In this way, we obtain for every $n$, a partition $\big(\Sigma^{(n)}_i\big)_{i \in I^n}$ and an associated partition of $T$ into pyramids $T_i^{(n)} = \conv(\Sigma_i^{(n)},\mathbf{v})$. In Section~\ref{sec: perron iterative} (see Theorem~\ref{theo: perron tree}) we explain how to iterate the shrinking process and prove (see Figure~\ref{fig:Perrons}) that there is a family of vectors $(\mathbf{w}_i)_{i \in I^n}$ such that
  \begin{equation*} \Big|\bigcup_{i \in I^n}\big( T_i^{(n)}+\mathbf{w_i}\big)\Big| \lesssim \frac{C}{n^{d-1}} \big|T\big|.
  \end{equation*}
  \begin{figure}[htbp]
    \centering
    \begin{subfigure}[b]{0.45\textwidth}
        \centering
        \includegraphics[width=0.7\textwidth]{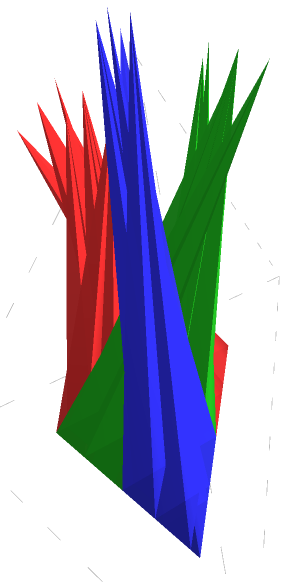}
        \caption{$n=2$}
    \end{subfigure}
    \hfill
    \begin{subfigure}[b]{0.45\textwidth}
        \centering
        \includegraphics[width=\textwidth]{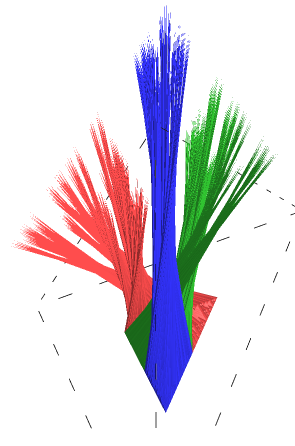}
        \caption{$n=3$}
    \end{subfigure}

    \vspace{0.5cm} 

    \begin{subfigure}[b]{0.45\textwidth}
        \centering
        \includegraphics[width=\textwidth]{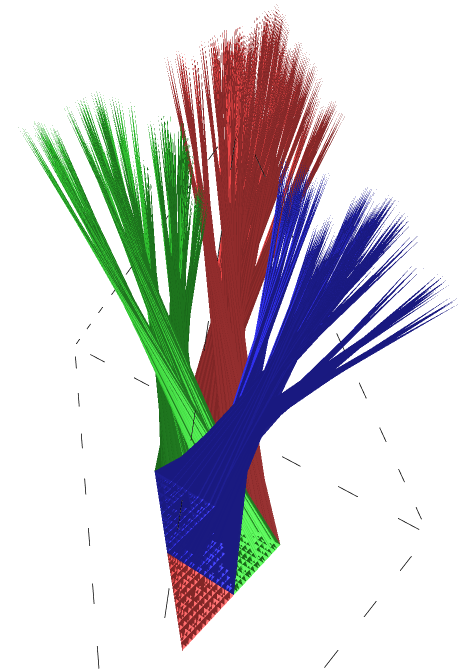}
        \caption{$n=4$}
    \end{subfigure}
    \hfill
    \begin{subfigure}[b]{0.45\textwidth}
        \centering
        \includegraphics[width=\textwidth]{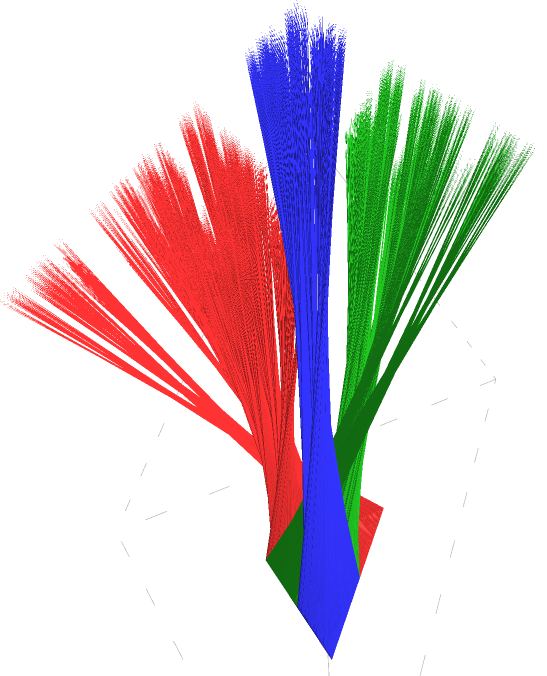}
        \caption{$n=5$}
    \end{subfigure}
    
    \caption{$n$-th iterated $3D$ Perron tree construction starting from a simplex, and different values of $n$.}
    \label{fig:Perrons}
\end{figure}

  All this goes in the correct direction, but it not enough to deduce Theorem~\ref{thm: main_kakeya} and Theorem~\ref{thm:main_tubes}. In order to do so, we will need more properties of the partition $(\Sigma_i)_{i \in I}$ that we start with. An additional sufficient property is that of a \emph{fair partition}, where all the pieces are congruent (by an isometry, or more generally a compact subgroup $G$ of $\mathrm{GL}(V)$) to a scaled down version of the initial piece. In dimension $1$, a $1$-simplex is just an interval, and for every integer $q$ it admits a unique fair subdivision into $q$ pieces. In dimension $2$, a fair subdivision of the simplex into $q^2$ pieces exists, but it is not unique (see Figure \ref{fig: subdivisions}). In higher dimension, it is not obvious, but true and classical, that fair subdivisions of the simplex exist (for example the Coxeter-Freudenthal-Kuhn subdivision); the barycentric subdivision is not fair. 
\begin{figure}[htbp]
\centering
    \makebox[0pt][c]{\hspace{3.5cm}
    \includegraphics[width=\textwidth]{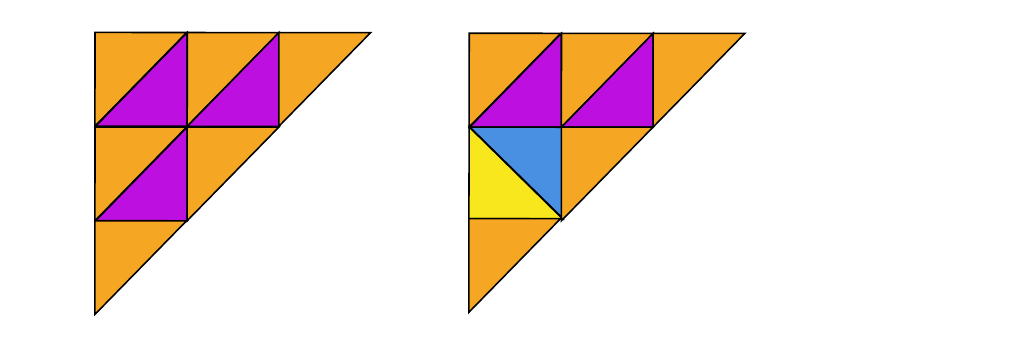}}
  \caption{Two different fair subdivisions of the 2-simplex with $q=3$.}
  \label{fig: subdivisions}
\end{figure}

  If we start with a fair subdivision and iterate it by choosing the identification of each piece of the partition with $\Sigma$  according to the compact group $G$, there are potentially many choices involved, but for every such choice this partitions $\big(\Sigma^{(n)}_i\big)_{i \in I^n}$ remain fair. So the shapes of the partition do not degenerate, in the sense that ratios are conserved. This fairness implies a crucial property for the pyramids $\conv(\Sigma_i^{(n)},\mathbf v)$: they all contain a $\delta$-tube for $\delta \geq c |I^n|^{-1/(d-1)}$, see Lemma~\ref{lemma: tubes inside perron} and Figure \ref{fig: tubes inside} below. Combined with the Perron tree construction we just described, this will lead to Theorem~\ref{thm: main_kakeya}.
\begin{figure}[htbp]
\centering
\makebox[0pt][c]{\hspace{9cm}\begin{minipage}{0.49\textwidth}
    \centering
    \includegraphics[width=\linewidth]{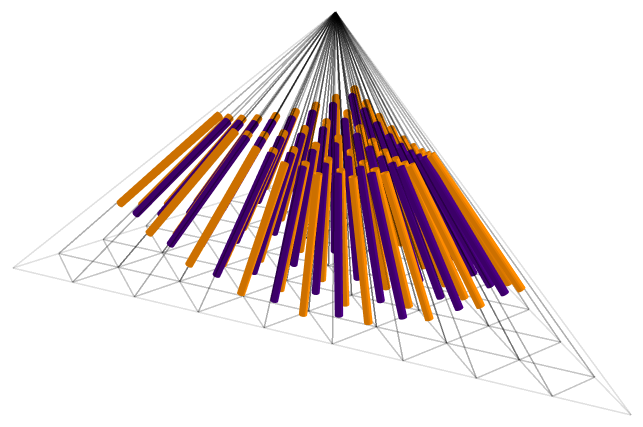}
\end{minipage}}\hfill
\makebox[0pt][c]{\hspace{-9cm}\begin{minipage}{0.49\textwidth}
    \centering
    \includegraphics[scale=0.5]{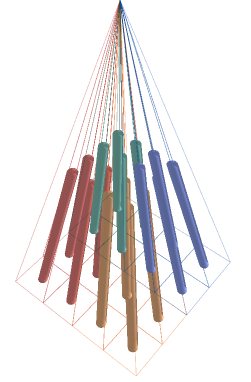}
\end{minipage}}
\caption{Particularization of Lemma \ref{lemma: tubes inside perron} to a fair subdivision of the $2$-simplex into 81 sub cells and for a cube with 16 sub cells, both with apexes at distance 3. One is able to contain a $\delta$-tube completely inside of each pyramid. All the tubes' top faces are distance of 1 to the apex and $\delta\sim 9^{-2}$ and $\delta\sim 4^{-2}$, respectively.}
\label{fig: tubes inside}
\end{figure}

If moreover we start with a subdivision that has an additional separation property (see Section~\ref{sec: perron base} for details), we can show that the tubes contained in the pyramids will have pairwise disjoint translates and deduce Theorem~\ref{thm:main_tubes}.

One may perform a Perron tree construction starting from any fair polytopal partition (as we will show), but the reader should have in mind the example of the cube and its natural dyadic subdivision, which makes the particularisation of the proofs in our construction cleaner and more intuitive. We thank Anthony Carbery who pointed this out to us.
\subsection*{Organization}  
Section \ref{sec: perron base}, studies the effect of translating the cells of the original polytope towards a distinguished point. The conclusion of this section is Proposition \ref{prop: base case}, the base step of the higher-dimensional Perron tree construction. Section \ref{sec: perron iterative} iterates this base step inductively to produce the full Perron tree construction, and derives a volume bound by an optimised choice of translation parameters at each step. Section \ref{sec: applications} applies the construction to prove Theorem \ref{thm:main_tubes} using specific families of overlapping $\delta$-tubes, and deduces Theorem \ref{thm: main_kakeya} on the existence of Kakeya sets with desirable $\delta$-neighbourhoods. Section \ref{sec: harmonic} records some consequences for harmonic analysis, in particular the proof of Corollary \ref{cor:radial_mult_sobolev} on Fourier multipliers and Besov spaces with logarithmic smoothness. At the end, we provide a link to a Jupyter notebook implementing the construction with interactive $3$D plots. In the Appendix \ref{appendix}, we provide the details on how our fully general construction is applicable to some non-trivial examples.
\subsection*{Acknowledgments}
The authors would like to warmly thank Anthony Carbery, Sam Craig, Oscar Dominguez, and Mark Lewko for their useful comments. The second author thanks Eduardo Tablate and Javier Parcet for their discussions on the ``logarithmic Bochner-Riesz conjecture".

\subsection*{Tool and computational resource disclosure} No AI system was used in the conception, proof, or writing of the mathematical content of this paper. Its use was restricted to generating code for the figures included herein and for the interactive visualizations in the accompanying GitHub repository, with selected portions of our construction provided as input.

\section{Definitions and notation}\label{sec: notation}

In the whole paper, $d\in \N_{\geq 2}$ will be used to denote an integer, and we consider the Euclidean space $\R^d$. We denote vectors and points in boldface, such as $\mathbf{v}\in \R^d$. By $\mathbf{e}_i$, we denote the canonical basis vectors $(0,\dots,1,\dots, 0)$. We write $[k]$ for the set $\{1,\dots,k\}$. For $\mathbf{p}\in \R^d$ and $\lambda\geq 0$, $\mathcal{H}_\mathbf{p}^\lambda$ will denote the homothety of center $\mathbf{p}$ and ratio $\lambda$; $\mathcal{R}_\mathbf{p}$ will denote the reflection across the point $\mathbf{p}$; $\operatorname{int}(X)$ will denote the topological interior of $X\subset \R^d$; $\overline{X}$ will denote the topological closure of $X\subset \R^d$; $|X|$ will denote the $d$-dimensional Lebesgue measure of a Lebesgue measurable subset $X\subset \R^d$. $\aff(X)$ will denote the affine span of $X\subset \R^d$:
\begin{equation*}
    \aff(X)=\left\{\sum_{i=1}^n\alpha_i\mathbf{x}_i\mid n\in \N,\;\mathbf{x}_i\in X,\;\alpha_i\in \R,\; \sum_{i=1}^n\alpha_i=1\right\}.
\end{equation*}

The direction of a non-empty subset $X\subset \R^d$ is:
\begin{equation*}
     \overrightarrow{X}= \operatorname{span}\{a-b\mid a,b\in X\}.
\end{equation*}

The opposite cone of a subset $A$ relative to a point $\mathbf{p}$ is the reflection with respect to $\mathbf{p}$ of the cone spanned by $A$ at $\mathbf{p}$:
\begin{equation}\label{eq:opposite_cone} \ocone{A}{\mathbf{p}} = \{ \mathbf{p} - \lambda(\mathbf{x}-\mathbf{p})\mid\mathbf{x} \in A, \lambda \geq 0\}.
\end{equation}

\subsection{Polytopes and partitions}
Consider a real vector space $V$ of dimension $d\in \N$, and let $m\in \N$ with $m\leq d$.
\begin{definition}
An $m$-dimensional {\bf polytope}, or $m$-polytope, is the convex hull of a finite set of points whose affine span has dimension $m$. 

 An $m$-dimensional {\bf pyramid}, or $m$-pyramid, is the convex hull of an $(m-1)$-polytope (the \emph{base}) and an affinely independent point (the \emph{apex}).
 
 An $m$-dimensional {\bf simplex}, or $m$-simplex, is the convex hull of $(m+1)$ affinely independent points. In particular, any $m$-simplex is an $m$-polytope and an $m$-pyramid whose base is a $(m-1)$-simplex. 

A {\bf polytopal subdivision} or {\bf subdivision} of an $m$-polytope $\Sigma$ is a collection $(\Sigma_i)_{i\in I}$ of $m$-polytopes (\emph{cells}) such that:
\begin{itemize}
    \item The union of the cells in $(\Sigma_i)_{i\in I}$ is the polytope $\Sigma$.
    \item For any two cells $\Sigma_1$, $\Sigma_2$, their intersection is either empty or a full face of a certain dimension shared by $\Sigma_1$, $\Sigma_2$.
\end{itemize}
A {\bf polytopal partition} or {\bf partition} of an $m$-polytope $\Sigma$ substitutes the second condition above by:
\begin{itemize}
    \item  For any two different cells $\Sigma_1$, $\Sigma_2$, their intersection has dimension $<m$.
\end{itemize}
In particular, a subdivision of an $m$-polytope is a partition too.
\end{definition}

As we explained in the introductions, there are different ways to break up a polytope into pieces. Some, that we call fair, are \emph{better} than the others.
\begin{definition} A {\bf fair subdivision} ({\bf fair partition}) of an $m$-polytope $\Sigma$ in $V$ is a subdivision (partition) $(\Sigma_i)_{i \in I}$ such that there is a compact subgroup $G \subset \mathrm{GL}(V)$ such that all the pieces $\Sigma_i$ are $G$-congruent to $\frac{1}{|I|^{\frac{1}{m}}}\Sigma$: i.e. there is $\mathbf{v}_i \in V$ and $g_i \in G$ such that
\begin{equation}\label{eq:G-congruence}\Sigma_i = \frac{1}{|I|^{\frac{1}{m}}} g_i(\Sigma) + \mathbf{v}_i.
\end{equation}
When this holds we say that the subdivision (partition) is $G$-fair.
\end{definition}

Since every compact subgroup of $\mathrm{GL}(V)$ preserves a scalar product, we can put a Euclidean structure on $V$ such that the subdivision is $O(V)$-fair. When $V$ is already equipped with a Euclidean structure (for example $V=\R^d$), we will always assume that the fair subdivision (partition) is $O(V)$-fair.

We record here the existence of fair subdivisions of the simplex:
\begin{theorem}\label{thm:CFK}\cite{Coxeter34,Freudenthal42} For every integer $q$, the $d$-simplex admits a fair subdivision into $q^d$ pieces.
\end{theorem}
\noindent We describe and study this subdivision in detail in Section~\ref{ap: simplex}.

We end this section with the following result which allows one to include congruent $\delta$-tubes into fair partitions of a given pyramid. We use the above convention that for a the base polytope $\Sigma\subset \R^d$, fair decomposition means $\mathrm{O}(d)$-fair.
\begin{lemma}\label{lemma: tubes inside perron}
Let $\Sigma \subset \R^d$ be a $(d-1)$-polytope and $\mathbf{v}\notin \operatorname{Aff}(\Sigma)$ such that $d\big(\mathbf{v},\operatorname{Aff}(\Sigma)\big)>2$. For all $\alpha>0$, there exists a $C=C_{\alpha,\Sigma,\mathbf{v}}>0$ such that for all fair partitions $(\Sigma_i)_{i\in [N]}$ of $\Sigma$, there exists $\big(R_i\big)_{i\in [N]}$ $\delta$-tubes with $\delta=CN^{-1/(d-1)}$,  such that:
\begin{enumerate}
    \item $R_i\subset (\operatorname{conv}(\Sigma_i,\mathbf{v}))$ $\forall\;i$,
    \item $\translate{R_i}{\alpha} \subset \ocone{\Sigma_i}{\mathbf{v}}$ $\forall\;i$.
\end{enumerate}
\end{lemma}
\begin{proof}
Let us first assume that $\alpha\in (0,1]$, and let $E=\aff(\Sigma)$.

By elementary trigonometry (see Figure~\ref{fig: prop diagram}), for all $\mathbf{x}\in \Sigma$ there exists a $C_1=C_1^\mathbf{x}>0$ such that if $\mathbf{x}'$ is the point on the segment $[\mathbf{x},\mathbf{v}]$ at distance $\alpha/2$ from $\mathbf{v}$, then $d(\mathbf{x}', \R^d\setminus \conv\big[B_E(\mathbf{x},\epsilon),\mathbf{v}\big])\geq C_1\epsilon$.  Moreover, by compactness of $\Sigma$, $C_1$ can be taken independent of $\mathbf{x}$. 

Now let $R>0$ such that $\Sigma$ contains a ball of radius $R$ (in $E$). Then for all $i\in [N]$, there exists a $\mathbf{x}_i\in \Sigma$ such that $\Sigma_i$ contains $B(\mathbf{x}_i,N^{-1/(d-1)}R)$. By the previous comment, $\conv(\Sigma_i,\mathbf{v})$ contains $R_i$, the $C_1N^{-1/(d-1)}R$-tube centered along $[\mathbf{x}_i,\mathbf{v}]$ and at distance $\alpha/2$ from $\mathbf{v}$.

Then we have $\translate{R_i}{\alpha}=\mathcal{R}_\mathbf{v}(R_i)$, which is contained in $\ocone{\Sigma_i}{\mathbf{v}}$. This proves the proposition for $\alpha\in (0,1]$.

 Finally, for $\alpha>1$, we apply the above argument for $\alpha=1$ to obtain $R_i\subset \conv(\Sigma_i,\mathbf{v})$ such that $\translate{R_i}{1}=\mathcal{R}_{\mathbf{v}}(R_i)\subset \ocone{\Sigma_i}{\mathbf{v}}$. If we translate the $\delta$-tube $\translate{R_i}{1}$ by the vector $(\alpha-1) \mathbf{u}(R_i)$, we get $\translate{R_i}{\alpha}$. But note that this translation of $\translate{R_i}{1}$ must remain contained  in $\ocone{\Sigma_i}{\mathbf{v}}$, as we wanted.
 \end{proof}
\begin{figure}[htbp]
    \centering
    \includegraphics[scale=.6]{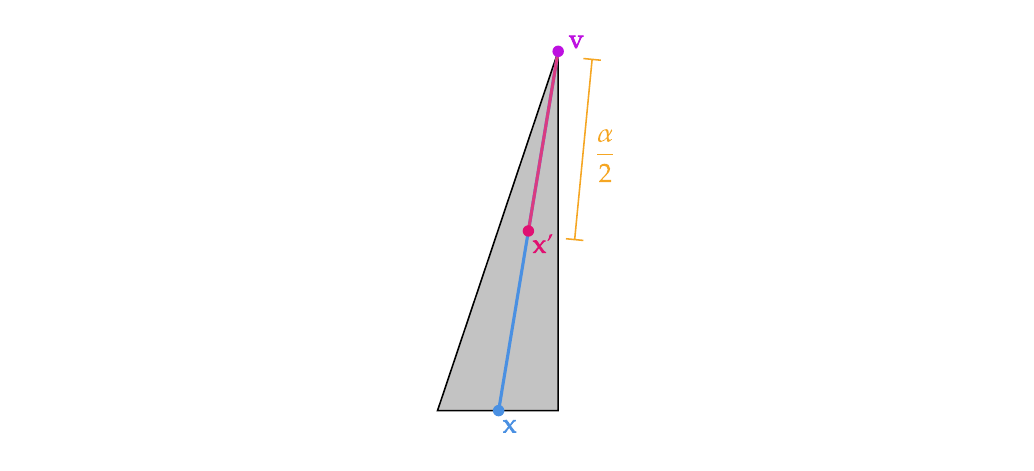}
    \caption{Illustrating the proof of Lemma \ref{lemma: tubes inside perron} when $d=2$.}
    \label{fig: prop diagram}
\end{figure}

The following lemma is proven in an analogous way to the previous one and will be a crucial ingredient in the proof of Theorem~\ref{thm: main_kakeya}:
\begin{lemma}\label{lemma: balls inside center of pyramid}
Let $\Sigma$ be a $(d-1)$-polytope and $\mathbf{v}\notin \operatorname{Aff}(\Sigma)$. There exists a $C=C_{\Sigma,\mathbf{v}}>0$ such that for every fair partition $(\Sigma_i)_{i\in [N]}$ of $\Sigma$, $C\conv(\Sigma_i,v)$ contains a translate of the $2\cdot N^{-1/(d-1)}$-neighbourhood of $\conv(\Sigma_i,v)$.
\end{lemma}

\section{The base case of the Perron tree construction}\label{sec: perron base}

In this whole section, we fix a $(d-1)$-dimensional polytope $\Sigma\subset \R^d$, a partition $(\Sigma_i)_{i \in I}$, a point $\mathbf{o}\in\aff(\Sigma)$ and a point $\mathbf{v} \in \R^d \setminus \aff(\Sigma)$. We consider the pyramids with distinguished apex $\mathbf v$ and with bases $\Sigma$ and $\Sigma_i$:
\[ T = \conv(\Sigma,\mathbf{v}),\qquad T_i = \conv(\Sigma_i,\mathbf{v}).\]

We assume moreover that there exists some $c\in (0,1)$ such that:
\begin{equation}\label{eq:assumption_shrinking} 
\boxed{\forall i\in I, \;\; \exists \;\mathbf{x}_i \in \Sigma,\;\;\Sigma_i \subset \mathcal{H}_{\mathbf{x}_i}^c(\Sigma)}
\end{equation}
We say that the family $(\Sigma_i,\mathbf{x}_i)_{i\in I}$ is affinely separated if for every $i \neq j$, there is a non-constant affine map $\varphi:\aff(\Sigma) \to \R$ such that $\varphi \geq 0$ on $\Sigma_i$, $\varphi \leq 0$ on $\Sigma_j$, and $\varphi(x_i)\geq \varphi(x_j)$. 

\begin{remark}
If $x_i \in \Sigma_i$, the family $(\Sigma_i,\mathbf{x}_i)_i$ is affinely separated (Hahn-Banach). But this is not always the case, as we shall explore in concrete examples.
\end{remark}

\begin{example}\label{ex:cubes} If $\Sigma=[0,1]^{d-1}$ is the cube and $(\Sigma_i)_{i \in \{0,1\}^{d-1}}$ is its standard fair subdivision $\Sigma_i = \prod_{k=1}^{d-1} [i_k/2,(1+i_k)/2]$, then \eqref{eq:assumption_shrinking} holds with $c=\frac 1 2$ and $\mathbf{x}_i=i$. It is also affinely separated, by the map $\varphi(\mathbf{y})=\mathbf{y}_k - \frac 1 2$ if $k$ is such that $i_k \neq j_k$.
\end{example}
We will also observe (see the Appendix \ref{ap: simplex}) that other partitions have the same properties:
\begin{example}\label{ex:simplex}
 If we take the Coxeter-Freudenthal-Kuhn fair subdivision of the $(d-1)$-simplex $\Sigma=\{\mathbf{y}\in \R^{d-1}\mid 0\leq y_1\leq \dots\leq y_{d-1}\leq q\}$ into $q^{d-1}$ pieces with $q \geq d$ (see Figure \ref{fig: MV subdiv}), then \eqref{eq:assumption_shrinking} holds with $c=\frac{q-1}{q}$. For a carefully chosen $\mathbf{x}_i$, it is also affinely separated (see Proposition \ref{lemma: cells embedding}).  \end{example}
\begin{example}  If $(\Sigma_i)_{i \in I}$ is any fair partition of a  polytope with $|I|>1$, then there is $k\in \N$ such that the repeated partition $(\Sigma_i)_{i \in I^k}$ satisfies \eqref{eq:assumption_shrinking} (see Proposition \ref{prop: fine polytope}).
\end{example}

The goal of this section is to understand what happens to the bases $\Sigma_i$ and the pyramids $T_i$ when we translate them towards the distinguished point $\mathbf{o}$. 

The first lemma studies this move at the level of the bases, and proves that we obtain a scaled down version of the original polytope. To make the construction precise, denote $\mathbf{u_i} = \mathbf{o}-\mathbf{x}_i$ the vector pointing from the \emph{special} point $\mathbf{x}_i$ associated to $\Sigma_i$ (which exists and which we will fix from now on together with $c$) towards $\mathbf{o}$. 
\begin{lemma}\label{lemma: base containment}
  For every $t \in (-\infty,1-c]$ and $i \in I$,
    \begin{equation*}\label{eq: base containment}
    \Sigma_i + t \mathbf{u}_i \subset \mathcal{H}_{\mathbf{o}}^{1-t}(\Sigma).
\end{equation*}
\end{lemma}
\begin{proof}
  By the expression
\[
\mathcal{H}^{\lambda}_{\mathbf{p}}(\mathbf{x}) = \lambda(\mathbf{x} - \mathbf{p}) + \mathbf{p} = \lambda \mathbf{x} + (1 - \lambda)\mathbf{p}
\]
for the homothety $\mathcal{H}^{\lambda}_{\mathbf{p}}$, the definition $\mathbf{u_i} = \mathbf{o}-\mathbf{x}_i$ and the assumption \eqref{eq:assumption_shrinking}, we have
\begin{align*}  \Sigma_i + t\mathbf{u}_i &\subset c \Sigma + (1-c) x_i +t(\mathbf{o}-\mathbf{x_i})\\
  & = t\mathbf{o} +  (1-t) \Big( \lambda \Sigma + (1-\lambda) \mathbf{x}_i\Big),
\end{align*}
for $\lambda = \frac{c}{1-t}$. By the assumption on $t \in (-\infty,1-c]$, we have that $\lambda \in [0,1]$. Moreover, $\mathbf{x}_i$ belongs to $\Sigma$, so by the convexity of $\Sigma$, $\lambda \Sigma + (1-\lambda) \mathbf{x}_i \subset \Sigma$ and the lemma is proved.
\end{proof}
\begin{figure}[htbp]
  \centering
  \begin{minipage}{0.3\textwidth}
    \includegraphics[width=\textwidth]{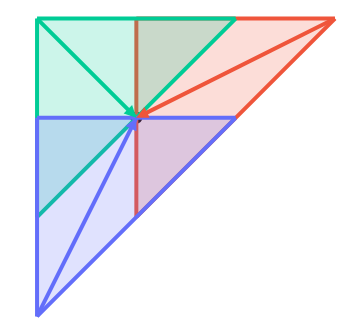}
  \end{minipage}
  \hfill
  \begin{minipage}{0.3\textwidth}
    \includegraphics[width=\textwidth]{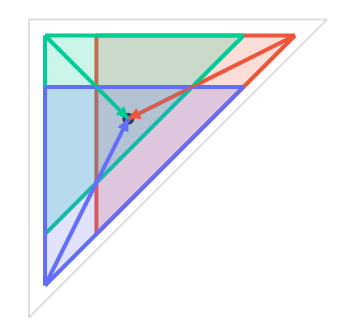}
  \end{minipage}
  \hfill
  \begin{minipage}{0.3\textwidth}
    \includegraphics[width=\textwidth]{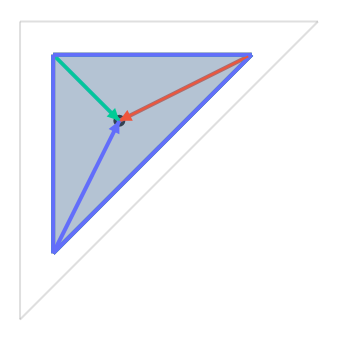}
  \end{minipage}
  \caption{Lemma \ref{lemma: base containment} for a $2$-simplex and $\mathbf{o}$ its barycenter.}
\end{figure}
\begin{figure}[htbp]
  \centering
  \begin{minipage}{0.3\textwidth}
    \includegraphics[width=\textwidth]{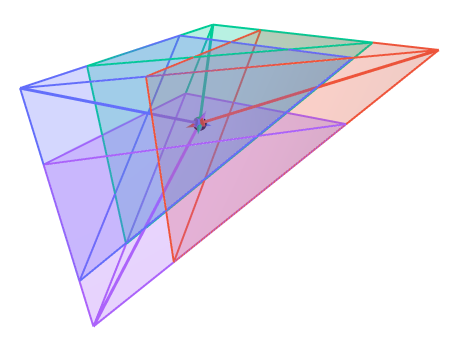}
  \end{minipage}
  \hfill
  \begin{minipage}{0.3\textwidth}
    \includegraphics[width=\textwidth]{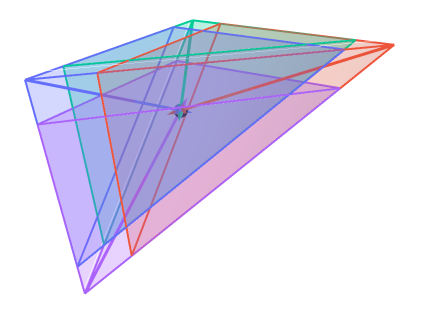}
  \end{minipage}
  \hfill
  \begin{minipage}{0.3\textwidth}
    \includegraphics[width=\textwidth]{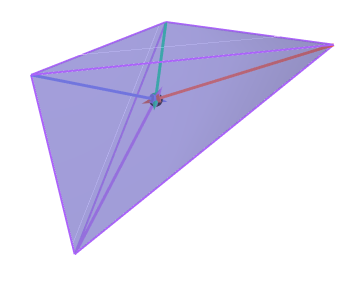}
  \end{minipage}
  \caption{Lemma \ref{lemma: base containment} for a $3$-simplex and $\mathbf{o}$ its barycenter.}
\end{figure}

The previous lemma motivates the introduction of a scaled down version of $T$ that we call the \emph{core}:
\[ C(t) = \mathcal{H}_\mathbf{o}^{1-t}\big(T\big).\]
The next lemma can be understood as the answer to the question: What is the minimum height at which the pyramids are no longer contained in the core when they are translated according to Lemma \ref{lemma: base containment}?
\begin{lemma}\label{lemma: lemma containment peaks}
We have that for $t\in [0,1-c]$:
\begin{equation*}
    (T_i+t \mathbf{u}_i)\setminus C(t)\subset \mathcal{H}_{\mathbf{v}}^{\frac{t}{1-c}}(T_i)+t\mathbf{u}_i.
\end{equation*}
\end{lemma}
\begin{proof}
We have \[ T_i + t \mathbf{u}_i = \bigcup_{\lambda \in [0,1]} \big(\lambda \mathbf{v} + (1-\lambda) (\Sigma_i + \frac{t}{1-\lambda} \mathbf{u}_i)\big)\]
and
\[ C(t) = \bigcup_{\lambda \in [0,1-t]} \big(\lambda \mathbf{v} + (1-t-\lambda)\Sigma + t\mathbf{o}\big) = \bigcup_{\lambda \in [0,1-t]} \big(\lambda \mathbf{v} + (1-\lambda) \mathcal{H}_{\mathbf{o}}^{\frac{1-t-\lambda}{1-\lambda}}(\Sigma)\big).\]
For $0\leq \lambda \leq 1-\frac{t}{1-c}$, we have $\frac{t}{1-\lambda} \leq 1-c$, so by Lemma~\ref{lemma: base containment} $\Sigma_i + \frac{t}{1-\lambda} \mathbf{u}_i \subset \mathcal{H}_{\mathbf{o}}^{1-\frac{t}{1-\lambda}}(\Sigma)$, and we deduce
\[  \lambda \mathbf{v} + (1-\lambda) (\Sigma_i + \frac{t}{1-\lambda} \mathbf{u}_i) \subset C(t).\]
For $1-\frac{t}{1-c}\leq \lambda \leq 1$,
\[  \lambda \mathbf{v} + (1-\lambda) (\Sigma_i + \frac{t}{1-\lambda} \mathbf{u}_i)
= \lambda \mathbf{v} + (1-\lambda) \Sigma_i + t\mathbf{u}_i \subset \mathcal{H}_{\mathbf{v}}^{\frac{t}{1-c}}(T_i) + t\mathbf{u}_i.\]
These two inclusions prove the lemma.
\end{proof}
\begin{figure}[htbp]
  \centering
  \begin{minipage}{0.3\textwidth}
    \includegraphics[width=\textwidth]{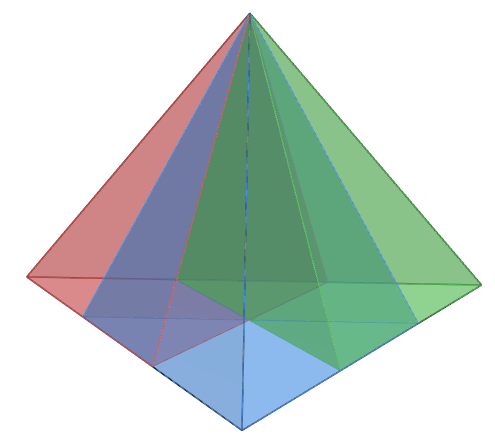}
  \end{minipage}
  \hfill
  \begin{minipage}{0.3\textwidth}
    \includegraphics[width=\textwidth]{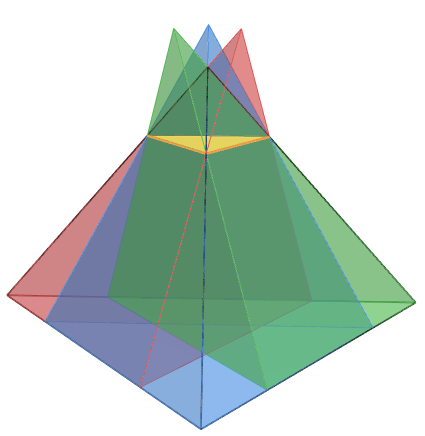}
  \end{minipage}
  \hfill
  \begin{minipage}{0.3\textwidth}
    \includegraphics[width=\textwidth]{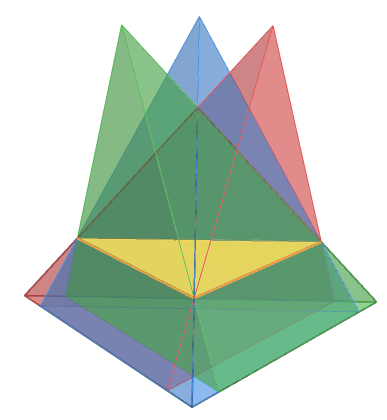}
  \end{minipage}
  \caption{Lemma \ref{lemma: lemma containment peaks} for a $2$-simplex and $\mathbf{o}$ its basis' barycenter. The coloured pyramids are the pyramids $T_i +t\mathbf{u_i}$ for increasing values of $t$, and the yellow section represents the height at which the pyramids are no longer completely contained in the core.}
\end{figure}

The next two lemmas describe the evolution of the opposite cones $\ocone{\Sigma_i}{\mathbf{v}}$ under the same translations as the pyramids.

\begin{lemma}\label{lemma: cone of directions} For every $i \in I$ and $t \in [0,1-c]$,
\begin{equation*}\label{eq: cone of directions}
    \ocone{\Sigma_i}{\mathbf{v}}+t\mathbf{u}_i\subset \ocone{C(t)}{(1-t)\mathbf v+t\mathbf{o}}=\mathcal{H}^{1-t}_\mathbf{o}(\ocone{\Sigma}{\mathbf{v}})
\end{equation*}
\end{lemma}
\begin{proof}
  By the definition \eqref{eq:opposite_cone}, we have to prove that
  \[ (1+\lambda)\mathbf v -\lambda \Sigma_i + t\mathbf{u}_i \subset \mathcal{H}_{\mathbf{o}}^{1-t}(\ocone{\Sigma}{\mathbf{v}})\]
for every $\lambda \geq 0$. By continuity we can assume $\lambda>0$. In that case, by Lemma~\ref{lemma: base containment}
\begin{align*} (1+\lambda) \mathbf v -\lambda \Sigma_i + t\mathbf{u}_i&= (1+\lambda)\mathbf v - \lambda( \Sigma_i - \frac{t}{\lambda}\mathbf{u}_i)\\
  & \subset (1+\lambda)\mathbf v - \lambda \mathcal{H}_{\mathbf{o}}^{1+\frac t \lambda}(\Sigma) = \mathcal{H}_{\mathbf{o}}^{1-t}\Big(\mathbf{v} - \frac{\lambda + t}{1-t}(\Sigma - \mathbf{v})\Big).
\end{align*}
  But $\mathbf{v} - \frac{\lambda + t}{1-t}(\Sigma - \mathbf{v}) \subset \ocone{\Sigma}{\mathbf{v}}$, so this concludes the proof of the lemma.
\end{proof}
\begin{figure}[htbp]
  \centering
  \begin{minipage}{0.3\textwidth}
    \includegraphics[width=\textwidth]{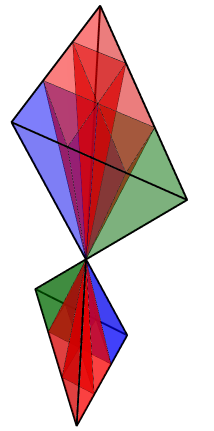}
  \end{minipage}
  \hfill
  \begin{minipage}{0.3\textwidth}
    \includegraphics[width=\textwidth]{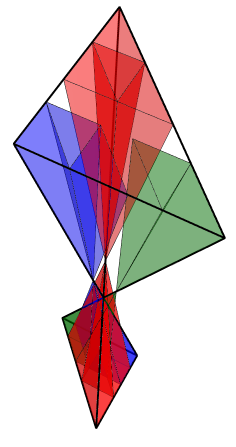}
  \end{minipage}
  \hfill
  \begin{minipage}{0.3\textwidth}
    \includegraphics[width=1.2\textwidth]{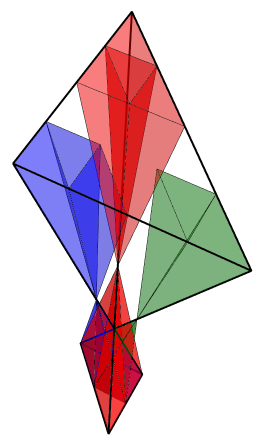}
  \end{minipage}
  \caption{Lemmas \ref{lemma: cone of directions} and \ref{lemma: non intersection}  for a $2$-simplex with fair subdivision into $9$  cells with $\mathbf{o}$ its basis' barycenter. The upper pyramids are the reflections of the lower ones (the $T_i$'s) across their respective apexes. $C(t)$ and its opposite cone are shown as black wire frames.}
  \label{fig: base case}
\end{figure}

\begin{lemma}\label{lemma: non intersection}
Assume that the family $(\Sigma_i,\mathbf{x}_i)_{i \in I}$ is affinely separated. For every $t\geq 0$ and every $i \neq j \in I$,
\begin{equation*}
    \Big(\ocone{\Sigma_i}{\mathbf{v}} + t \mathbf{u}_i\Big) \cap \Big(\ocone{\Sigma_j}{\mathbf{v}} + t \mathbf{u}_j\Big)
\end{equation*}
has dimension $<d$.
\end{lemma}
\begin{proof}
Let $\varphi: \aff(\Sigma) \to \R$ be the non constant affine map separating $(\Sigma_i,\mathbf{x}_i)$ and $(\Sigma_j,\mathbf{x}_j)$. Extend $\varphi$ affinely to  $\aff(\Sigma,\mathbf{v})$ by setting $\varphi(\mathbf{v})=0$. Then 
\[ \varphi(\ocone{\Sigma_i}{\mathbf{v}} + t \mathbf{u}_i) = \ocone{\varphi(\Sigma_i)}{\varphi(\mathbf{v})} + t (\varphi(\mathbf{o}) - \varphi(\mathbf{x}_i)) = (-\infty,t (\varphi(\mathbf{o}) - \varphi(\mathbf{x}_i))],\]
and 
\[ \varphi(\ocone{\Sigma_j}{\mathbf{v}} + t \mathbf{u}_j) = [t (\varphi(\mathbf{o}) - \varphi(\mathbf{x}_j)),\infty),\]
so \[
  \Big(\ocone{\Sigma_i}{\mathbf{v}} + t \mathbf{u}_i\Big) \cap \Big(\ocone{\Sigma_j}{\mathbf{v}} + t \mathbf{u}_j\Big) \subset \{ \mathbf{z} \in \aff(\Sigma,\mathbf{v})=\R^d \mid \varphi(\mathbf{z})=t(\varphi(\mathbf{o}) - \varphi(\mathbf{x}_i))\}.\]
  Note that the intersection can even be empty if $\varphi(x_j)<\varphi(x_i)$, but this case still verifies the containment above. The right hand set is a $(d-1)$-dimensional affine space, concluding the proof.
\end{proof}

Here is the summary of what we have proved:
\begin{proposition}\label{prop: base case}
Let $\Sigma \subset \R^d$ be a  $(d-1)$-dimensional polytope with a partition $(\Sigma_i)_{i \in I}$ satisfying \eqref{eq:assumption_shrinking}. 

$ \forall\;\mathbf{v}\notin \operatorname{Aff}(\Sigma)$, $\forall\;\mathbf{o}\in \operatorname{Aff}(\Sigma)$ and $\forall\;t\in (0,1-c]$, if $T=\operatorname{conv}(\Sigma, \mathbf{v})$, $T_i=\operatorname{conv}(\Sigma_i, \mathbf{v})$, $\widetilde{T}=\ocone{\Sigma}{\mathbf{v}}$ and $\widetilde{T}_i=\ocone{\Sigma_i}{\mathbf{v}}$, then there exist vectors $(\mathbf{w}_i)_{i\in I}\subset \overrightarrow{\operatorname{Aff}(\Sigma)}$ such that:
\begin{enumerate}[(i)]
    \item $\displaystyle \bigg|\bigg[\bigcup_{i\in I} T_i+\mathbf{w}_i\bigg]\setminus \mathcal{H}^{1-t}_\mathbf{o}(T)\bigg|\leq \frac{t^d|T|}{(1-c)^d} $.
    \item $\displaystyle \big(\widetilde{T}_i+\mathbf{w}_i\big)_{i\in I}$ are contained in $\mathcal{H}_\mathbf{o}^{1-t}(\widetilde{T})$, and have pairwise disjoint interiors if the family $(\Sigma_i,\mathbf{x}_i)_{i \in I}$ is affinely separated.
    \item The $\Sigma_i+
    \mathbf{w}_i$ are contained in $\mathcal{H}_\mathbf{o}^{1-t}(\Sigma)$ $\forall\;i$.
\end{enumerate}
\end{proposition}
\begin{proof}
We prove the result with $\mathbf{w}_i = t\mathbf{u}_i$.

Let us first prove (i). By Lemma \ref{lemma: lemma containment peaks}, using that a homothety of factor $\lambda$ scales the Lebesgue measure by $\lambda^d$, we have
\begin{equation*}
     \bigg|\bigg[\bigcup_{i\in I} 
     T_i+\mathbf{w}_{i}\bigg]\setminus \mathcal{H}^{1-t}_\mathbf{o}(T)\bigg|\leq \sum_{i\in I}|\big(T_i+\mathbf{w}_{i}\big)\setminus \mathcal{H}^{1-t}_\mathbf{o}(T)|
    \leq \sum_{i\in I} \frac{t^d}{(1-c)^d}|T_i|= \frac{t^d|T|}{(1-c)^d} .
\end{equation*}

The first part of (ii) is an application of Lemma~\ref{lemma: cone of directions}, and the second part is Lemma~\ref{lemma: non intersection}.

(iii) is a direct consequence of the fact that, thanks to Lemma~\ref{lemma: base containment}
\[\Sigma_i+\mathbf{w}_i= \Sigma_{i}+t\mathbf{u}_{i}\subset \mathcal{H}^{1-t}_\mathbf{o}(\Sigma).\qedhere\]
\end{proof}

\section{The Perron tree construction}\label{sec: perron iterative}
\subsection{The induction argument}
We now explain how to iterate Proposition~\ref{prop: base case}. The standing assumption in this section is the following: $\Sigma \subset \R^d$ is a $(d-1)$-polytope  with a partition $(\Sigma_i)_{i \in I}$ satisfying \eqref{eq:assumption_shrinking}, and with the additional property that every piece $\Sigma_i$ is the image of $\Sigma$ by an affine map. 

This additional property holds for example for a fair partition (in which case we will always choose the affine map in $|I|^{1-d} O(d)$), but also for partitions of the simplex into simplices. It allows to define, for every $n \geq 1$, the $n$-step iterated partition $(\Sigma_i^{(n)})_{i \in I^n}$.
\begin{proposition}\label{prop: perron tree}

Let $\mathbf{v}\notin \operatorname{Aff}(\Sigma)$, $\mathbf{o}\in \operatorname{Aff}(\Sigma)$ and $t_1, t_2,\dots,t_n\in (0,1-c]$. Write  $P_m=\prod_{j=1}^{m}(1-t_j)$ for $m\in \N$, $T=\operatorname{conv}(\Sigma, \mathbf{v})$, $T_i=\operatorname{conv}(\Sigma_i^{(n)}, \mathbf{v})$, $\widetilde{T}=\ocone{\Sigma}{\mathbf{v}}$ and $\widetilde{T}_i=\ocone{\Sigma_i^{(n)}}{\mathbf{v}}$. Then there exists vectors $(\mathbf{w}_i)_{i\in I^n}\subset \overrightarrow{\operatorname{Aff}(\Sigma)}$ such that:
\begin{enumerate}[(i)]
    \item $\displaystyle \bigg|\bigg[\bigcup_{i\in I^n} T_i+\mathbf{w}_i\bigg]\setminus\mathcal{H}_\mathbf{o}^{P_n}(T)\bigg|\leq \frac{|T|}{(1-c)^d}\sum_{i=1}^n\bigg(\prod_{j=1}^{i-1}(1-t_j)^d\bigg)t_{i}^d$.
    \item $\displaystyle \big(\widetilde{T}_i+\mathbf{w}_i\big)_{i\in I^n}$ are contained in $\mathcal{H}_\mathbf{o}^{P_n}(\widetilde{T})$, and have pairwise disjoint interiors if the family $(\Sigma_i,\mathbf{x}_i)_{i \in I}$ is affinely separated.
\item The $\Sigma_i^{(n)}+
    \mathbf{w}_i$ are contained in $\mathcal{H}_\mathbf{o}^{P_n}(\Sigma)$ $\forall\;i$.
\end{enumerate}
\end{proposition}
\begin{proof}
   We will proceed by induction in the number of iterations $n\in \N$. The base case $n=1$ is exactly Proposition~\ref{prop: base case}. Assume that the statement holds up to $(n-1)$, and let us prove it for $n$:
   
  Apply the proposition for $n=1$, to $\Sigma$, $\mathbf{v}$, $\mathbf{o}$ and $t_n$ to obtain vectors $(\mathbf{w}_{i_n}')_{i_{n}\in I}$ that satisfy, if  $T_{i_n} = \conv(\Sigma_{i_n},\mathbf{v})$ and $\widetilde{T}_{i_n}=\ocone{\Sigma_{i_n}}{ \mathbf{v}}$, the following conclusions:
\begin{enumerate}[(a)]
    \item\label{item:induction11} $\displaystyle \bigg|\bigcup_{i_n\in I} \bigg(T_{i_n}+\mathbf{w}_{i_n}'\bigg)\setminus \mathcal{H}_\mathbf{o}^{1-t_n}(T)\bigg|\leq \frac{|T|t_n^d}{(1-c)^d}$,
    \item\label{item:induction12} $\displaystyle \big(\widetilde{T}_{i_n}+\mathbf{w}_{i_n}'\big)_{i_n\in I}$ are contained in $\mathcal{H}^{1-t_n}(\widetilde{T})$, and have pairwise disjoint interiors if the family $(\Sigma_i,\mathbf{x}_i)_{i \in I}$ is affinely separated.
    \item\label{item:induction13} $\Sigma_{i_n}+\mathbf{w}_{i_n}'$ are contained in $\mathcal{H}^{1-t_n}_\mathbf{o}(\Sigma)$ $\forall \;i_n\in I$.
\end{enumerate} 

Let us re-index the $n$-iterated partition $\big(\Sigma_i^{(n)}\big)_{i \in I^n}$ as $\big(\Sigma_{i'}^{i_n}\big)_{i'\in I^{n-1},i_n \in I}$ in such a way that for every fixed $i_{n}\in I$, $\big(\Sigma_{i'}^{i_n}\big)_{i'\in I^{n-1}}$ is an $(n-1)$-iterated partition of $\Sigma_{i_n}$. By the induction hypothesis we can apply the proposition for $(n-1)$, to the simplex $\Sigma_{i_{n}}$, with $\mathbf{v}$, $\mathbf{o}$ and $t_1$, ..., $t_{n-1}\in (0,1-c]$ to obtain vectors $\big(\mathbf{w}_{i'}^{i_n}\big)_{i'\in I^{n-1}}\subset \overrightarrow{\operatorname{Aff}(\Sigma_{i_n})}=\overrightarrow{\operatorname{Aff}(\Sigma)}$ such that:
\begin{enumerate}[(a),resume]
    \item\label{item:inductionn1} $\displaystyle \bigg|\bigcup_{i'\in I^{n-1}} \bigg(\operatorname{conv}(\Sigma_{i'}^{i_n}, \mathbf{v})+\mathbf{w}_{i'}^{i_n}\bigg)\setminus \mathcal{H}_\mathbf{o}^{P_{n-1}}(T_{i_n})\bigg|\leq \frac{|T_{i_n}|}{(1-c)^d}\sum_{k=1}^{n-1} P_{k-1}^dt_k^d$,
    \item\label{item:inductionn2} $\displaystyle \big(\ocone{\Sigma^{i_n}_{i'}}{ \mathbf{v}}+\mathbf{w}^{i_n}_{i'}\big)_{i'\in I^{n-1}}$ are contained in $\mathcal{H}^{P_{n-1}}_\mathbf{o}(\widetilde{T}_{i_n})$, and have pairwise disjoint interiors if the family $(\Sigma_i,\mathbf{x}_i)_{i \in I}$ is affinely separated.
    \item\label{item:inductionn3} The $\Sigma_{i'}^{i_n}+
    \mathbf{w}^{i_n}_{i'}$ are contained in $\mathcal{H}^{P_{n-1}}_\mathbf{o}(\Sigma_{i_n})$ $\forall \;i'\in I^{n-1}$.
\end{enumerate}

With all the above, define for $i=(i',i_n)\in I^n=I\times I^{n-1}$:
\begin{equation*}
 \mathbf{w}_i=\mathbf{w}^{i_n}_{i'}+P_{n-1}\mathbf{w}_{i_n}'\in \overrightarrow{\operatorname{Aff}(\Sigma)}\qquad.
\end{equation*}
Let us check that $(\mathbf{w}_i)_{i \in I^n}$ satisfy the conclusion of the proposition. 

In order to prove property (i), using the following inclusion, valid for any sets,
\[ \bigcup_{i_n} (A_{i_n} \setminus C) \subset \Big(\bigcup_{i_n} (A_{i_n} \setminus B_{i_n})\Big) \cup \Big(\bigcup_{i_n} (B_{i_n} \setminus C)\Big),\]
to $A_{i_n} = \bigcup_{i' \in I^{n-1}} (T_{i',i_n}+\mathbf{w}_{i'}^{i_n}+P_{n-1}\mathbf{w}_{i_n}')$, $B_{i_n} = \mathcal{H}_{\mathbf{o}}^{P_{n-1}}(T_{i_n} + \mathbf{w}_{i_n}')$ and $C= \mathcal{H}_{\mathbf{o}}^{P_{n}}(T)$, we obtain
\begin{multline*}
    \Big| \bigcup_i (T_i + \mathbf{w}_i)\setminus \mathcal{H}_{\mathbf{o}}^{P_n}(T) \Big| \leq 
\sum_{i_n} \Big|\bigcup_{i'} (T_{i',i_n} + \mathbf{w}_{i'}^{i_n} + P_{n-1} \mathbf{w}_{i_n}')\setminus \mathcal{H}_{\mathbf{o}}^{P_{n-1}}(T_{i_n} + \mathbf{w}_{i_n}')\Big|\\
+ \Big| \bigcup_{i_n} \mathcal{H}_{\mathbf{o}}^{P_{n-1}}\big((T_{i_n} + \mathbf{w}_{i_n}') \setminus \mathcal{H}_{\mathbf{o}}^{1-t_n}(T)\big)\Big|.
\end{multline*} 
We have $\mathcal{H}_{\mathbf{o}}^{P_{n-1}}(T_{i_n} + \mathbf{w}_{i_n}') = \mathcal{H}_{\mathbf{o}}^{P_{n-1}}(T_{i_n}) + P_{n-1}\mathbf{w}_{i_n}'$, so by the translation-invariance of the Lebesgue measure and \ref{item:inductionn1} the first term is
\[\sum_{i_n} \Big|\bigcup_{i'} (T_{i',i_n} + \mathbf{w}_{i'}^{i_n})\setminus \mathcal{H}_{\mathbf{o}}^{P_{n-1}}(T_{i_n})\Big| \leq \sum_{i_n\in I} \frac{ |T_{i_n}|}{(1-c)^d}\sum_{k=1}^{n-1} P_{k-1}^dt_k^d = \frac{ |T|}{(1-c)^d}\sum_{k=1}^{n-1} P_{k-1}^dt_k^d.\]
Since the homothety $\mathcal{H}_{\mathbf{o}}^{P_{n-1}}$ scales the Lebesgue's measure by the factor $P_{n-1}^d$, the second term is
\[ P_{n-1}^d \Big| \bigcup_{i_n} \big( (T_{i_n} + \mathbf{w}_{i_n}') \setminus \mathcal{H}_{\mathbf{o}}^{(1-t_n)}(T)\big)\Big|,\]
which by \ref{item:induction11} is less than $P_{n-1}^d|T| t_n^d(1-c)^{-d}$. Altogether this proves (i).

In order to prove (ii), if the family $(\Sigma_i,\mathbf{x}_i)_{i \in I}$ is affinely separated, take $i\neq j\in I^n$. If $i_n=j_n$, then $i'\neq j'$ and by \ref{item:inductionn2} 
\begin{equation*}
   \Big( \operatorname{int}(\ocone{\Sigma_i^{(n)}}{ \mathbf{v})}+\mathbf{w}_{i'}^{i_n}\Big)\cap \Big( \operatorname{int}(\ocone{\Sigma_j^{(n)}}{ \mathbf{v})}+\mathbf{w}_{j'}^{j_n}\Big)=\emptyset,
\end{equation*}
or equivalently since $\mathbf{w}_i = \mathbf{w}_{i'}^{i_n} + P_{n-1} \mathbf{w}_{i_n}'$ and $\mathbf{w}_j = \mathbf{w}_{j'}^{j_n} + P_{n-1} \mathbf{w}_{i_n}'$,
\begin{equation}\label{eq:cones_disjoint_interiors}
    \Big( \operatorname{int}(\widetilde{T}_i)+\mathbf{w}_{i}\Big)\cap \Big( \operatorname{int}(\widetilde{T}_j)+\mathbf{w}_{j}\Big)=\emptyset .
\end{equation}

If $i_n\neq j_n$, since by \ref{item:inductionn2} we have that
\begin{equation*}
    \operatorname{int}(\widetilde{T}_i)+\mathbf{w}_{i} = \Big( \operatorname{int}(\ocone{\Sigma_{i'}^{i_n}}{ \mathbf{v})}+\mathbf{w}_{i'}^{i_n}\Big)+P_{n-1}\mathbf{w}_{i_n}'
   \subset \operatorname{int}\Big(\mathcal{H}^{P_{n-1}}_\mathbf{o}(\widetilde{T}_{i_n}+\mathbf{w}_{i_n}')\Big),
\end{equation*}
and by \ref{item:induction12}, the two sets $\mathcal{H}^{P_{n-1}}_\mathbf{o}(\widetilde{T}_{i_n}+\mathbf{w}_{i_n}')$ and $\mathcal{H}^{P_{n-1}}_\mathbf{o}(\widetilde{T}_{j_n}+\mathbf{w}_{j_n}')$ have disjoint interiors, we deduce that \eqref{eq:cones_disjoint_interiors} holds also in this case.

Regardless of the fact that the partition is affinely separated, the previous containment identities also imply that for any $i\in I^n$:
\begin{equation*}
   \widetilde{T}_i+\mathbf{w}_{i} \subset \mathcal{H}^{P_{n-1}}_\mathbf{o}(\widetilde{T}_{i_n}+\mathbf{w}_{i_n}')\subset \mathcal{H}^{P_{n-1}}_\mathbf{o}( \mathcal{H}^{1-t_n}_\mathbf{o}(\widetilde T)) =  \mathcal{H}^{P_n}_\mathbf{o}(\widetilde{T}),
\end{equation*}
where the inclusions are due to \ref{item:inductionn2} and \ref{item:induction12} respectively.

(iii) is also immediate: by \ref{item:inductionn3}, we have
\begin{equation*}
    \Sigma_i^{(n)}+\mathbf{w}_i=\big(\Sigma_{i'}^{i_n}+\mathbf{w}_{i'}^{i_n}\big)+P_{n-1}\mathbf{w}_{i_n}'\subset \mathcal{H}^{P_{n-1}}_\mathbf{o}(\Sigma_{i_n})+P_{n-1}\mathbf{w}_{i_n}'=\mathcal{H}^{P_{n-1}}_\mathbf{o}(\Sigma_{i_n} + \mathbf{w}_{i_n}'),
\end{equation*}
which by \ref{item:induction13} is contained in $\mathcal{H}^{P_{n-1}}_\mathbf{o}(\mathcal{H}^{1-t_n}_\mathbf{o}(\Sigma))=\mathcal{H}^{P_n}_\mathbf{o}(\Sigma)$.
\end{proof}

The following set of figures illustrate our Perron tree construction for $d=3$, $\Sigma$ a $2$-simplex, $\mathbf{v}=(1,2,1)$, $\mathbf{o}=(1,2,0)$ and $n=2:$
\begin{figure}[htbp]
  \centering
  \begin{minipage}{0.3\textwidth}
    \includegraphics[width=\textwidth]{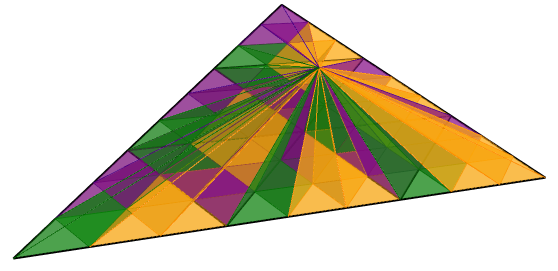}
  \end{minipage}
  \hfill
  \begin{minipage}{0.3\textwidth}
    \includegraphics[width=\textwidth]{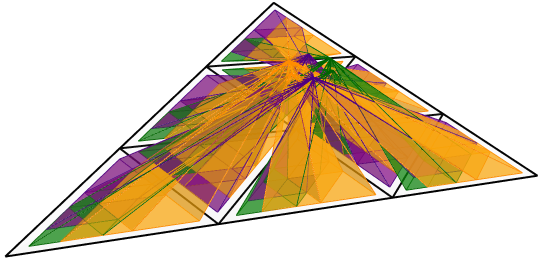}
  \end{minipage}
  \hfill
  \begin{minipage}{0.3\textwidth}
    \includegraphics[width=\textwidth]{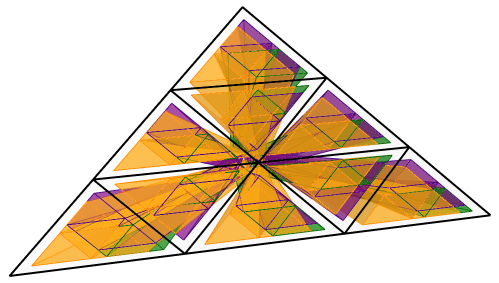}
  \end{minipage}
  \caption{First iteration: The first image shows a fair partition of a $2$-simplex into $81$ cells, which were obtained by applying a fair decomposition to the original simplex, and then successively to each of the sub cells formed after the first partition. Image two shows how we apply the basic construction to each group of 9 sub sub cells, taking the parent sub cell as our starting simplex. The third image is a bottom view of the construction, and allows us to better understand how the basic construction has been applied to each group of sub sub cells.}\label{fig: first iteration}
\end{figure}

\begin{figure}[htbp]
  \centering
  \begin{minipage}{0.3\textwidth}
    \includegraphics[width=\textwidth]{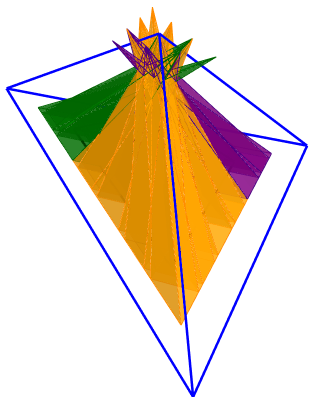}
  \end{minipage}
  \hfill
  \begin{minipage}{0.3\textwidth}
    \includegraphics[width=\textwidth]{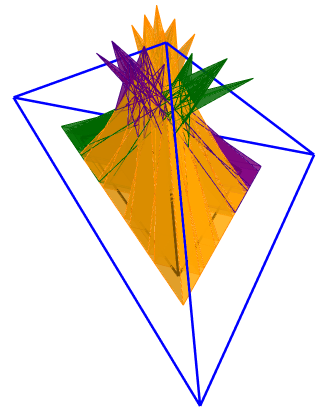}
  \end{minipage}
  \hfill
  \begin{minipage}{0.3\textwidth}
    \includegraphics[width=\textwidth]{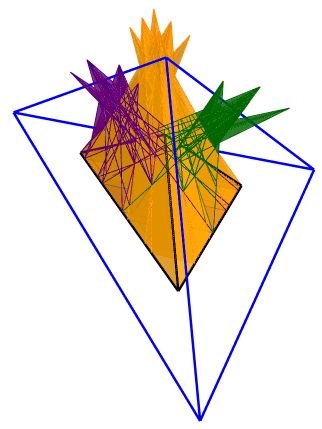}
  \end{minipage}
  \caption{Second iteration: In the last figure of \ref{fig: first iteration}, we see that we are left with 9 Perron trees, whose bases are all congruent to a separated scaled-down version of the fair simplicial decomposition of 9 cells. In the first image above, we have translated towards the barycenter of the pyramid's base these Perron trees as a block, until their bases just touch. In the second image we observe how we carry out again the basic construction, now using the 9 cores of each Perron tree as our starting pyramids (note the change in the colours of the cores in order to identify the translations defined in the basic construction applied to this second step); we remark that all of the pyramids in the same lower-level Perron tree are translated as a block in this iteration. The third image shows the final $n=2$ Perron tree resulting from our construction, where the black wire frame is the $\mathcal{H}_\mathbf{o}^{P_2}(T)$ core in the theorem.}\label{figure: second iteration}
\end{figure}
\begin{figure}[htbp]
  \centering
  \begin{minipage}{0.3\textwidth}
    \includegraphics[width=\textwidth]{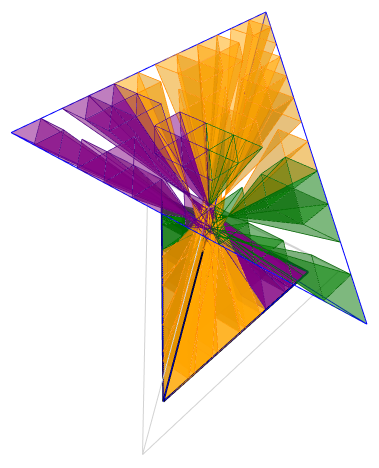}
  \end{minipage}
  \hfill
  \begin{minipage}{0.3\textwidth}
    \includegraphics[width=\textwidth]{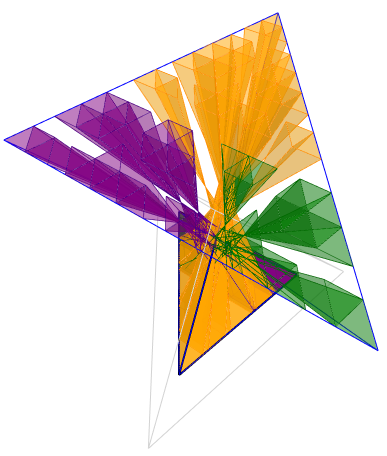}
  \end{minipage}
  \hfill
  \begin{minipage}{0.3\textwidth}
    \includegraphics[width=\textwidth]{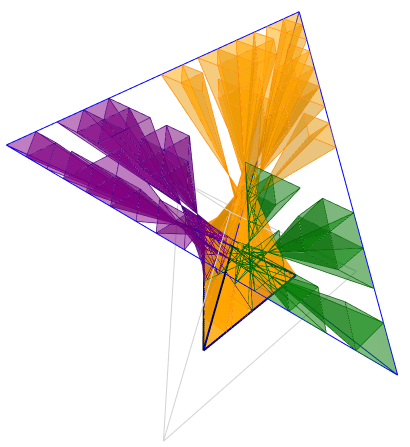}
  \end{minipage}
  \caption{Reflection of the pyramids in the construction for different values of $t_2\in (0,1/3]$, where the respective $t_2$ is that of the figure in \ref{figure: second iteration} immediately right above. Notice how all the opposite cones have disjoint pairwise interiors, and how the two-level iteration gives rise to the self replicating distribution of the reflections. The blue wire frame is the opposite cone of directions $\mathcal{H}_\mathbf{o}^{P_2}(\widetilde{T})$, so all reflected pyramids are contained within it. } 
\end{figure}

\subsection{The main construction}
With the same notation as in Proposition~\ref{prop: perron tree}, we obtain the bound
\begin{multline}\label{eq: volume after construction}
    \bigg|\bigcup_{i\in I^n} (T_i+\mathbf{w}_i)\bigg|\leq\bigg|\bigg[\bigcup_{i\in I^n} (T_i+\mathbf{w}_i)\bigg]\setminus \mathcal{H}_\mathbf{o}^{P_n}(T)\bigg|+|\mathcal{H}_\mathbf{o}^{P_n}(T)|\leq\\ \frac{|T|}{(1-c)^d}\sum_{i=1}^n P_{i-1}^d t_{i}^d+P_n^d|T|
    =|T|\bigg[\frac{1}{(1-c)^d}\sum_{i=1}^n(P_{i-1} - P_i)^d+P_n^d\bigg].
\end{multline}

We could choose all the $t_i$'s to be equal, but it is better to choose them increasing. A convenient (and optimal, see Remark~\ref{remark:n1-doptimal}) choice is to set, for $i\in[n]$,
\begin{equation}\label{eq: specific ti}
    t_i=\frac{1}{n-i+(1-c)^{-1}}.
\end{equation}
First of all note that $t_i\in (0,1-c]$ for all $i\in [n]$, so this is indeed a valid choice. Moreover, since
\begin{equation*}
    1-t_i=\frac{n-i-1+(1-c)^{-1}}{n-i+(1-c)^{-1}} = \frac{t_i}{t_{i+1}},
\end{equation*}
the expression for $P_m$ simplifies:
\begin{equation*}
   P_m= \prod_{j=1}^m (1-t_j)  = \prod_{j=1}^m \frac{t_j}{t_{j+1}}= \frac{t_1}{t_{m+1}}.    
\end{equation*}

Plugging the above into \eqref{eq: volume after construction}, we get the volume bound:
\begin{equation*}
     \bigg|\bigcup_{i\in I^n} (T_i+\mathbf{w}_i)\bigg|\leq |T|\bigg[
     \frac{1}{(1-c)^d} n t_1^d + \frac{t_1^d}{t_{n+1}^d}
     \bigg]\lesssim_{d,\Sigma}\frac{|T|}{n^{d-1}},
\end{equation*}
where the dependence of the bound on $c\in (0,1)$ is ultimately governed by the geometry of $\Sigma$ through condition \eqref{eq:assumption_shrinking}.

The key properties of the above construction are thus summarized in the next theorem:
\begin{theorem}[The d-dimensional Perron tree construction]\label{theo: perron tree}
Let $\Sigma \subset \R^d$ and $(\Sigma_i)_{i\in I}$ as before. For every $n$, there exists a family of vectors $\big(\mathbf{w}_i\big)_{i\in I^n}\subset \overrightarrow{\operatorname{Aff}(\Sigma)}$ such that, for every $\mathbf{v} \in \R^d\setminus \aff(\Sigma)$, the family $T_i = \conv(\Sigma_i^{(n)},\mathbf{v})$ forms a partition of the $d$-pyramid $T=\conv(\Sigma,v)$ satisfying the volume bound:
\begin{equation}\label{eq: general simplex volume}
    \bigg|\bigcup_{i\in I^n}(T_i+\mathbf{w}_i)\bigg|\lesssim_{d,\Sigma} \frac{|T|}{n^{d-1}},
\end{equation}
and such that the cones $\{\ocone{\Sigma_i^{(n)}}{\mathbf{v}}+\mathbf{w}_i\}$ are pairwise disjoint if the family $(\Sigma_i,\mathbf{x}_i)_{i \in I}$ is affinely separated.
\end{theorem}
\begin{proof}
As we have already justified, the theorem is immediate from Proposition~\ref{prop: perron tree} for $\mathbf{o} \in \aff(\Sigma)$ arbitrary and $t_1,\dots,t_n$ defined by \eqref{eq: specific ti}. 
\end{proof}

\begin{definition}
With previous notation, we will refer to the set
\begin{equation*}
    \bigcup_{i\in I^n}\Big( T_i+\mathbf{w}_i\Big)
\end{equation*}
constructed in Theorem \ref{theo: perron tree} as the $n$-iteration Perron tree construction of $T=\conv(\Sigma, \mathbf{v})$ relative to the $n$-step iterated partition $(\Sigma_i^{(n)})_{i \in I^n}$.
\end{definition}
\begin{remark}For the particularisation of $\Sigma=Q=[0,1]^{d-1}$, $\mathbf{o}=\mathbf{0}$ and any $\mathbf{v}\notin \R^{d-1}\times\{0\}$, one can check that the cells in the $n$-iteration dyadic decomposition of $Q$ can be expressed as
\begin{equation*}
    Q_i=\mathbf{x}_i+2^{-n}[0,1]^{d-1},\qquad \mathbf{x}_i=\sum_{j=1}^n\boldsymbol{\beta}_j2^{-j},\qquad \boldsymbol\beta_j\in \{0,1\}^{d-1}.
\end{equation*}
From the representation above, it is straightforward to see that the vectors $\{\mathbf{w}_i\}_{i\in [2^{d-1}]^n}$ of the $n$-th iteration Perron tree construction given by
\begin{equation*}
    \mathbf{w}_i=-\frac{1}{n+1}\sum_{j=1}^n\boldsymbol{\beta}_j2^{-j}(n-j+2)=-\sum_{j=1}^n\boldsymbol\beta_j2^{-j}\Big(1-\frac{j-1}{n+1}\Big)
\end{equation*}
correspond to the overall translation applied to each pyramid with base $Q_i$ and apex  $\mathbf{v}$.

Up to a change of variables, these formulas are the same as those provided by Keich in \cite{RefWorks:keich1999lp}, thus showcasing that we have been able to generalise his construction for higher dimensions. 

We see in particular that, if $T^{n,d}\subset \R^d$ is the $n$-iteration Perron tree in dimension $d$ associated to the dyadic decomposition of the cube with apex $\mathbf{v}=(0,\dots,0,1)$, and if for $h\in [0,1]$, $T^{n,d}_h$ is its slice at height $h$:
\[ T^{n,d}_h=\{\mathbf{x}\in \R^{d-1}\mid(\mathbf{x},h)\in T^{n,d}\},\]
then $T^{n,d}_h=(T^{n,2}_h)^{d-1}$. Keich proved that $\sup_{h\in [0,1]}|T_h^{n,2}|< \frac{1}{n}$, so we deduce 
\[|T^{n,d}|=\int_0^1 |T^{n,d}_h|\;dh=\int_0^1 |T^{n,2}_h|^{d-1}\;dh\leq \frac{1}{n^{d-1}}.\]
This provides an alternative proof of the estimate in Theorem~\ref{theo: perron tree} in the particular case of the tree whose base is the dyadic decomposition of the cube.
\end{remark}
\begin{figure}[htb]
    \centering
    \begin{subfigure}[b]{0.45\textwidth}
        \centering
        \includegraphics[width=\textwidth]{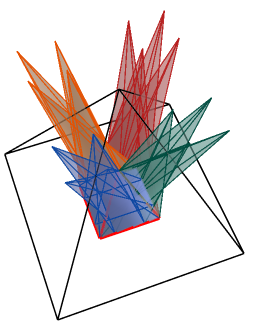}
        \caption{$n=2$}
    \end{subfigure}
    \hfill
    \begin{subfigure}[b]{0.45\textwidth}
        \centering
        \includegraphics[width=\textwidth]{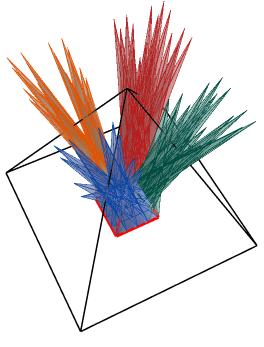}
        \caption{$n=3$}
    \end{subfigure}

    \vspace{0.5cm} 

    \begin{subfigure}[b]{0.45\textwidth}
        \centering
        \includegraphics[width=\textwidth]{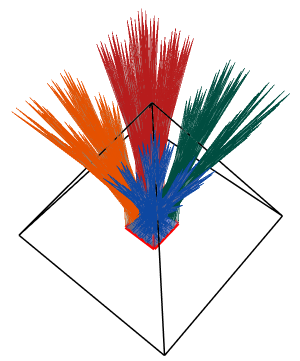}
        \caption{$n=4$}
    \end{subfigure}
    \hfill
    \begin{subfigure}[b]{0.45\textwidth}
        \centering
        \includegraphics[width=1.1\textwidth]{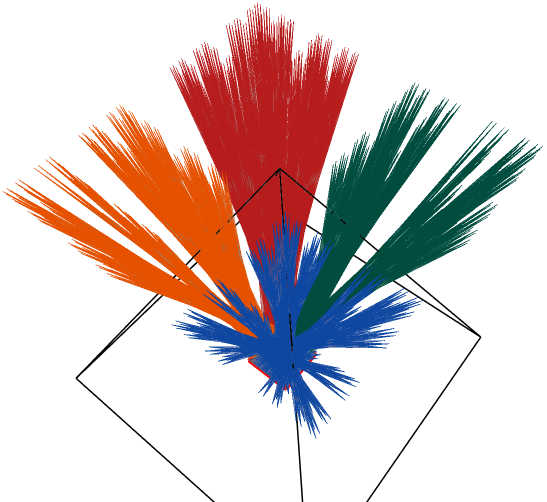}
        \caption{$n=5$}
    \end{subfigure}
    
    \caption{$n$-th iterated Perron tree construction for $d=3$, $\Sigma$ the unit cube, and different values of $n$.}
    \label{fig: Perrons cubes}
\end{figure}
\begin{remark}\label{remark:n1-doptimal}
The $O (n^{1-d})$ asymptotic obtained is the best one can do if one employs equation \eqref{eq: volume after construction} to bound the our Perron tree construction volume  above. Indeed, by Hölder's inequality,
\begin{align*} 1 &= \bigg(\sum_{i=1}^n P_{i-1} - P_i\bigg) + P_n
 \leq \bigg( P_n^d + \sum_{i=1}^n (P_{i-1} - P_i)^d \bigg)^{\frac 1 d} (n+1)^{\frac{d-1}{d}},
\end{align*}
which implies that the right hand side of equation \eqref{eq: volume after construction} is at least
\[ \gtrsim
\frac{|T|}{(n+1)^{d-1}}.\]
\end{remark}

\section{Applications to Kakeya-type problems} \label{sec: applications}
\subsection{Volumes of higher dimensional intersecting tubes}\label{subsection:tubes}
We now prove Theorem~\ref{thm:main_tubes}. 

For every $\delta\in (0,1)$, we define the number $f_d^\alpha(\delta)$ as the infimum of the real numbers $\varepsilon>0$ such that there exists a finite family of $\delta$-tubes $\{R_j\}$ such that the $\translate{R_j}{\alpha}$ are pairwise disjoint but
    \begin{equation}\label{eq: f delta}
        \frac{\big|\bigcup_j R_j\big|}{\sum_j|R_j|}\leq \varepsilon.
    \end{equation}
    
Keich \cite{RefWorks:keich1999lp} proves that there is a constant $C>0$ such that 
\begin{equation*}
    \frac{C^{-1}}{\log 2/\delta}\leq f_2^1(\delta)\leq \frac{C}{\log 2/\delta}.
\end{equation*}
 Theorem~\ref{thm:main_tubes} is equivalent to the inequality 
 \begin{equation}\label{eq:main_tubes} f_d^\alpha(\delta) \leq \frac{C_{d,\alpha}}{|\log \delta|^{d-1}}.\end{equation}
 This is where our Perron tree construction becomes useful, and where we need fair partitions.

\begin{proof}[Proof of Theorem~\ref{thm:main_tubes}]
As explained before, we have to prove \eqref{eq:main_tubes}.

Let $\Sigma$ be a $(d-1)$-polytope $\Sigma \subset \R^d$ with a fair partition $(\Sigma_i)_{i \in I}$ with $|I|>1$ satisfying \eqref{eq:assumption_shrinking}, and take $\mathbf{v}$ such that $d(\mathbf{v},\aff(\Sigma))>2$. We also require that the family $(\Sigma_i,\mathbf{x}_i)_{i \in I}$ is affinely separated. The proof works for any polytopal partition satisfying these conditions, but for simplicity let us write it for the dyadic subdivisions of the unit cube.

Given an arbitrary $n\in \N$, let $(\Sigma_i)_{i \in [2^{n(d-1)}]}$ be the $n$-iteration dyadic decomposition and $(\mathbf{w}_i)_{i \in [2^{n(d-1)}]}$ be the vectors given by Theorem~\ref{theo: perron tree}. By Lemma~\ref{lemma: tubes inside perron} each cube-based pyramid $T_i$ contains a $\delta_n$-tube $R_i$ whose translate $\translate{R_i}{\alpha}$ in contained in the interior of $\ocone{\Sigma_i}{\mathbf{v}}$, for
\[ \delta_n = C_{\alpha,d} [2^{n(d-1)}]^{-1/(d-1)} = C_{\alpha,d} 2^{-n}.\] Then the $\delta_n$-tubes $R_i + \mathbf{w}_i$ have the property that their translates $\translate{R_i}{\alpha}$ are pairwise disjoint, because they are contained in the interiors of the reflected cones $\ocone{\Sigma_i}{\mathbf{v}}+\mathbf{w}_i$, which are pairwise disjoint by Theorem~\ref{theo: perron tree}. 
As a consequence, we have that 
\begin{equation*}
    f^\alpha_d(\delta_n)\leq \frac{\bigg|\bigcup_{i\in [2^{d-1}]^n}\big(R_i+\mathbf{w}_i\big)\bigg|}{\sum_{i\in [2^{d-1}]^n}|R_i+\mathbf{w}_i|}\leq \frac{\bigg|\bigcup_{i\in [2^{d-1}]^n}\big(\operatorname{conv}(\Sigma_i,\mathbf{v})+\mathbf{w}_i\big)\bigg|}{\sum_{i\in [2^{d-1}]^n}C_{d,\alpha}'2^{-n(d-1)}}\lesssim_{\alpha,d} \frac{n^{1-d}}{2^{n(d-1)}2^{-n(d-1)}}=n^{1-d}.
\end{equation*}

Now, for every $\delta\in (0,C_{\alpha,d}]$, there exists a unique $n_0\in \N$ such that $\delta_{n_0+1}\leq \delta< \delta_{n_0}$. Since $f^\alpha_d(\delta)$ satisfies the identity:
\begin{equation*}
    f_d^\alpha(\delta)\leq \bigg(\frac{\delta'}{\delta}\bigg)^{d-1}f_d^\alpha(\delta')\qquad 0<\delta\leq \delta',
\end{equation*}
we then have
\begin{equation*}
    f_d^\alpha(\delta)\leq \bigg(\frac{\delta_{n_0}}{\delta_{n_0+1}}\bigg)^{d-1}f_d^\alpha(\delta_{n_0})\lesssim_{\alpha,d}n_0^{1-d}\lesssim_{\alpha,d}|\log\delta|^{1-d},
\end{equation*}
where in the last line we have used that $|\log\delta|\lesssim_{\alpha,d} n_0$.
\end{proof}

\begin{figure}[H]
  \centering
  \makebox[0pt][c]{\hspace{9cm}\begin{minipage}{0.3\textwidth}
    \includegraphics[width=\textwidth]{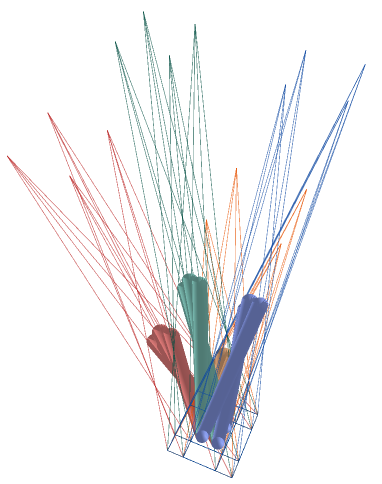}
  \end{minipage}}
  \hfill
  \makebox[0pt][c]{\hspace{-9cm}\begin{minipage}{0.4\textwidth}
    \includegraphics[width=\textwidth]{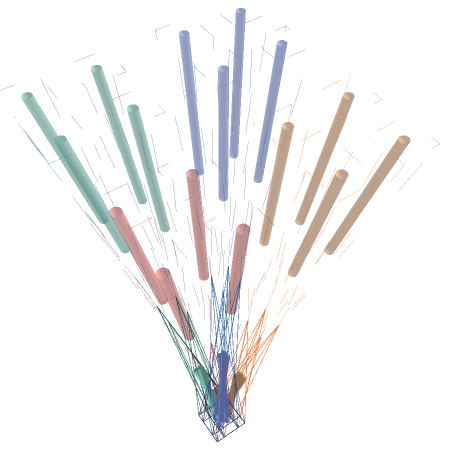}
  \end{minipage}}
  \caption{Illustration of the proof of Theorem \ref{thm:main_tubes} for $d=3$. The first image shows the tubes translated as a block within their respective pyramid during the 2-iterations Perron tree construction of the cube. The second image includes the reflected pyramids across their respective apex, which by construction contain the $\translate{R_i}{\alpha}$'s.}
\end{figure}
\begin{remark}
The first author's interest in obtaining bounds for the functions $f_d^\alpha$ is what motivated the search for a generalisation of the 2 dimensional Perron tree construction. Thus, in order to be able to insert congruent $\delta$-tubes in each $d$-dimensional cell of $\operatorname{conv}(\Sigma,\mathbf{v})$,  the hypothesis of a fair partition of $\Sigma$ was crucial. For example, it is clear that Lemma \ref{lemma: tubes inside perron} is false for partitions of $\Sigma$ in which the cells do not have a bounded aspect ratio, i.e. the barycentric subdivision of a simplex. 
\end{remark}

\subsection{Higher dimensional Kakeya sets}
\begin{definition}[Kakeya set]
A Kakeya set in $E\subset \R^d$ is a compact set which contains a
unit length line segment in every direction of $\mathbb S^{d-1}\subset \R^d$. That is, given $\mathbf{e}\in \mathbb S^{d-1}$, there exists an $\mathbf{x}\in \R^d$ such that
\begin{equation*}
    \{\mathbf{x}+s\mathbf{e}\mid s\in [-1/2,1/2]\}\subset E.
\end{equation*}
\end{definition}
\begin{lemma}\label{lemma: lemma kakeya}For every $\delta\in (0,1)$, there exists a Kakeya set $G^\delta$ such that $|\overline{\mathcal{N}_{2\delta}(G^\delta)}|\lesssim_d|\log\delta|^{1-d}$, where $\mathcal{N}_\delta(X)$ denotes the $\delta$-neighbourhood of $X\subset \R^d$.
\end{lemma}
\begin{proof}
Note that it suffices to prove the lemma for a set $G^\delta$ containing an open set of directions $S$ in $\mathbf{S}^{d-1}$ (the same for all $\delta$), rather than a Kakeya set. Indeed, there is a finite number of rotated copies of $G^\delta$ that will contain every direction in $\mathbb{S}^{d-1}$, and their union verifies the lemma.

It also suffices to prove the lemma with $\delta \in \{ 2^{-n} \mid n \in \N\}$: for a general $\delta$ the lemma follows from the case of $2^{-n}$ for $n$ such that $2^{-n-1} \leq \delta \leq 2^{-n}$.

Set $\Sigma = [0,1]^{d-1}$, seen in $\R^{d-1}\times\{0\} \subset \R^{d}$ and $\mathbf{v}=(\mathbf{0},1)$. Let $C$ be the constant given by Lemma \ref{lemma: balls inside center of pyramid}. We will prove the above claim for
\[ S = \Big\{ \frac{\mathbf{v} - \mathbf{x}}{\|\mathbf{v} - \mathbf{x}\|} \mid \mathbf{x} \in \Sigma\Big\}.\]

We perform the Perron tree construction for $n$ steps, and scale it by $C$. We obtain a fair subdivision $\big(\Sigma_i\big)_{i\in [2^{d-1}]^n}$ of $\Sigma$ and vectors $\big(\mathbf{w}_i\big)_{i\in [2^{d-1}]^n}$ such that
\[ \bigcup_{i \in  [2^{d-1}]^n} \big(C \conv(\Sigma_i,\mathbf{v}) + C\mathbf{w}_i\big)\]
has measure $O(n^{-(d-1)})$. But by Lemma~\ref{lemma: balls inside center of pyramid}, each of the pyramids $C \conv(\Sigma_i,\mathbf{v}) + C\mathbf{w}_i$ contains a translate (say by the vector $\mathbf{w}'_i$) of the $2\cdot 2^{-n}$-neighbourhood of the pyramid $\conv(\Sigma_i,\mathbf{v})$. So if we set
\[ G^{2^{-n}} = \bigcup_{i \in  [2^{d-1}]^n} \big(\conv(\Sigma_i,\mathbf{v}) + \mathbf{w}'_i\big)\]
we see that $G^{2^{-n}}$ contains a unit segment parallel to $\mathbf v - \mathbf x$ for every $\mathbf{x} \in \bigcup_i \Sigma_i=\Sigma$, that is a unit segment parallel to every element of $S$. And its $2\cdot 2^{-n}$-neighbourhood has measure $O(n^{-(d-1)}) = O((|\log 2^{-n}|)^{-(d-1)})$.
\end{proof}

\begin{proposition}\label{prop: prop kakeya} Let $f:(0,1)\rightarrow (0,\infty)$ non-decreasing and consider the following two statements:
\begin{itemize}
    \item[$(1)_f$] $\forall \delta\in (0,1)$ there exists a Kakeya set $G^\delta$ such that $|\overline{\mathcal{N}_{2\delta}(G^\delta)}|\leq f(\delta)$.
    \item[$(2)_f$] There exists a Kakeya set $E$ such that $\forall \delta\in (0,1)$ $|\mathcal{N}_\delta(E)|\leq f(\delta)$.
\end{itemize}
There is a dimension dependent constant $C_d\geq 1$ such that $(1)_f$ implies $(2)_{\widetilde{f}}$ for $\widetilde{f}(\delta)=C_d f(C_d\sqrt[4]{\delta})$ and $\delta\in (0,C_d^{-4})$.
\end{proposition}
The proof follows the classical patching argument in dimension $2$, see \cite{RefWorks:wolff2003recent,RefWorks:keich1999lp}. We state a useful lemma in the proof.
\begin{lemma}\label{lem:local_kakeya} Assume $(1)_f$. For every unit segment $I\subset \R^d$ with direction $\mathbf{u}$, for every $0<\delta \leq \varepsilon \leq \frac 1 2$, there is a subset $F^{I,\varepsilon,\delta}$ of $\R^d$ such that:
\begin{enumerate}
    \item $F^{I,\varepsilon,\delta}$ contains a unit segment of direction $\mathbf{v}\in \mathbf{S}^{d-1}$ for all $||\mathbf{v}-\mathbf{u}||\leq \varepsilon/2$.
    \item $F^{I,\varepsilon,\delta}$ is contained in the $\sqrt{d-1} \varepsilon$-neighbourhood of $I$.
    \item $ |\overline{\mathcal{N}_{2\delta}(F^{I,\varepsilon,\delta})}|\lesssim_d \varepsilon^{d-1}f(\delta/\varepsilon)$.
\end{enumerate} 
\end{lemma} 
\begin{proof}
  Without loss of generality, by applying a rigid motion, we can assume that the unit segment $I$ is $[-\mathbf{e}_1/2,\mathbf{e}_1/2]$ and its direction is $\mathbf{u}=\mathbf{e}_1$.

  Let $S_\varepsilon = \{\mathbf{v} \in \mathbf{S}^{d-1} \mid \langle \mathbf{v},\mathbf{e}_1\rangle \geq \frac{1}{\sqrt{1+\varepsilon^2}}\}$. Observe that, for a unit vector $\mathbf{v}$ with decomposition $\mathbf{v} = v_1 \mathbf{e}_1+ \mathbf{v}'$ and $\mathbf{v}'$ orthogonal to $\mathbf{e}_1$,
  \begin{equation}\label{eq:characterization_Sepsilon} \mathbf{v} \in S_\varepsilon \iff \varepsilon v_1 \geq \|\mathbf{v'}\|.
  \end{equation}
  We claim that for every $0<\delta<1$, there is a subset $H^\delta$ of the unit square $[-1/2,1/2]^d$ such that:
  \begin{itemize}
  \item for every $\mathbf v \in S_{1/2}$, $H^\delta$ contains a segment parallel to $\mathbf v$ joining the side $\{x_1=-1/2\}$ to the side $\{x_1=1/2\}$,
    \item $|\overline{\mathcal{N}_{\delta}(H^{\delta})}|\lesssim_d f(\delta)$.
  \end{itemize}
  Taking this claim for granted, the lemma will follow by defining $F^{I,\varepsilon,\delta}$ as the image of $H^{\delta/\varepsilon}$ by the map $T:(x_1,\mathbf{x}')\mapsto (x_1, 2 \varepsilon \mathbf{x}')$. Then $F^{I,\varepsilon,\delta}$ is contained in $[-1/2,1/2]\times [-\varepsilon,\varepsilon]^{d-1}$, which is contained in the $\sqrt{d-1} \varepsilon$ neighbourhood of $I$. The choice of $S_\varepsilon$ was dictated by the fact (immediate from \eqref{eq:characterization_Sepsilon}) that $\{\frac{Tv}{\|Tv\|}\mid v \in S_{1/2}\}$ is precisely $S_{\varepsilon}$. But a small computation gives that for a unit vector, $\|\mathbf{v}-\mathbf{e}_1\|\leq \varepsilon/2$ implies $\mathbf{v} \in S_{\varepsilon}$. So for every such $\mathbf{v}$, $F^{I,\varepsilon,\delta}$ contains a segment parallel to $\mathbf{v}$ joining the side $\{x_1=-1/2\}$ to the side $\{x_1=1/2\}$, so in particular a unit length segment parallel to $\mathbf{v}$. Finally, the volume estimate is clear because $T$ has determinant $(2\varepsilon)^{d-1}$ and $T$ contracts the distances by at most $2\varepsilon$:
  \[ |\overline{\mathcal{N}_{2\delta}(F^{I,\varepsilon,\delta})}| \leq |T \overline{\mathcal{N}_{\delta/\varepsilon}(H^{\delta/\varepsilon})}| = (2\varepsilon)^{d-1}|\overline{\mathcal{N}_{\delta/\varepsilon}(H^{\delta/\varepsilon})}| \lesssim_d \varepsilon^{d-1} f(\delta/\varepsilon).\]

  Let us prove the claim. For every $\mathbf{k} \in \Z^d$, consider the cube $Q_{\mathbf{k}}= \frac{\mathbf{k}}{2} + [-1/2,1/2]^d$, with its distinguished sides $Q_k^\pm= \big(\frac{\mathbf{k}}{2} + \{\pm \frac 1 2\mathbf e_1\}\big)\times [-1/2,1/2]^{d-1}$. We have a covering $\R^d= \bigcup_{\mathbf{k}\in \Z^d} Q_{\mathbf{k}}$, that is not too redundant in the sense that there is a number $N_d$ such that every point of $\R^d$ belongs to the $1$-neighbourhood of at most $N_d$ different cubes. For example $N_d=7^d$ works. However, the covering has enough redundancy so that for every segment $J$ of length $c_0=3\sqrt{5}/4$ and parallel to a vector in $S_{\frac 1 2}$, there is a cube $Q_k$ such that $Q_k \cap J$ joins $Q_k^-$ to $Q_k^+$. Indeed, the projection of $J$ on the axis $\R \mathbf{e}_1$ has length at least $c_0/\sqrt{1+1/2^2} = \frac 3 2$, so it contains entirely a segment of the form $k_1/2 + [-1/2,1/2]$ for some integer $k_1$. Let $J'$ be the intersection of $J$ with the band $\big(k_1/2 + [-1/2,1/2]\big) \times \R^{d-1}$. Since $\mathbf{v}$ belongs to $S_{\frac 1 2}$, by \eqref{eq:characterization_Sepsilon} the projection of $J'$ on $\{0\}\times \R^{d-1}$ has length $\leq 1/2$, so it is contained in a cube of the form $\mathbf{k'}/2+[-1/2,1/2]^{d-1}$ for some $\mathbf{k'} \in \Z^{d-1}$. Equivalently, for $\mathbf{k}=(k_1,\mathbf{k'})$, $J'$ is contained in $Q_{\mathbf{k}}$ and joins $Q_k^-$ to $Q_k^+$. We can now define
  \[ H^\delta = \bigcup_{\mathbf{k} \in \Z^d}\Big( (Q_{\mathbf{k}} \cap c_0 G^{\delta/c_0}) - \mathbf{k}/2\Big).\]
    It is clear from the definition that $H^\delta$ is contained in $Q_{\mathbf{0}}=[-1/2,1/2]^d$. Since $G^{\delta/c_0}$ is a Kakeya set, for every direction $\mathbf{v} \in S_{\frac 1 2}$, $c_0 G^{\delta/c_0}$ contains a vector parallel to $\mathbf{v}$ of length $c_0$, so by the preceding discussion $H^\delta$ contains a vector parallel to $\mathbf{v}$ joining $Q_{\mathbf{0}}^-$ to $Q_{\mathbf{0}}^+$. Finally, for every point $\mathbf{x}$ in the $\delta$-neighbourhood of $H^\delta$, there is at least one point $\mathbf x'$ in the $\delta$-neighbourhood of $c_0 G^{\delta/c_0}$ such that $\mathbf{x}-\mathbf{x'} \in \Z^d/2$. Moreover the map $\mathbf{x}\mapsto \mathbf{x'}$ is at-most-$N_d$-to-one, so we have the bound
    \[ |\overline{\mathcal{N}_{\delta}(H^\delta)}| \leq N_d |\overline{\mathcal{N}_{\delta}(c_0 G^{\delta/c_0})}| \leq N_d c_0^d f(\delta/c_0) \leq N_d c_0^d f(\delta).\]
This proves the claim and the lemma.        
\end{proof}

\begin{proof}[Proof of Proposition~\ref{prop: prop kakeya}] For the proof of the proposition, given $\delta_n=2^{-2^n}$, we want to define a sequence of sets $F_n$ such that:
\begin{itemize}
    \item[(a)] $F_n$ is a Kakeya set for every $n\in \N$.
    \item[(b)] $\mathcal{N}_{\delta_n}(F_n)\subset \mathcal{N}_{\delta_{n-1}}(F_{n-1})$ for all $n\in\N$.
    \item[(c)] $|\overline{\mathcal{N}_{2\delta_n}(F_n)}|\lesssim_d f(C_d\sqrt{\delta_n})$, for some dimension dependent $C_d\geq1$.
\end{itemize}
We proceed inductively. We set $F_0=G^{\frac 1 2}$. If $F_0,\dots, F_{n-1}$ are constructed, define $\varepsilon_n = \frac{\delta_{n-1} - \delta_n}{\sqrt{d-1}}$, and take $A$ a maximal $\varepsilon_n/2$ separated subset of $\mathbf{S}^{d-1}$, so that $\#|A|\leq C_d \varepsilon_n^{-(d-1)}$. For each $\mathbf{u}\in A$, choose a unit segment $I_\mathbf{u}$ parallel to $\mathbf{u}$ in $F_{n-1}$, which can be done since $F_{n-1}$ is a Kakeya set. We now define
\begin{equation*}
    F_n=\bigcup_{\mathbf{u}\in A}F^{I_\mathbf{u},\varepsilon_{n},\delta_n}
\end{equation*}
where the $F^{I,\varepsilon,\delta}$ are as in Lemma~\ref{lem:local_kakeya}. For this definition to make sense, we need $\varepsilon_n \geq\delta_n$, that is $n \geq 1+ \log_2 \log_2(1+\sqrt{d-1})$. Otherwise, we just set $F_n=F_{n-1}$. We have to check that $F_n$ satisfies the three properties we ask for. This is obvious when  $n < 1+ \log_2 \log_2(1+\sqrt{d-1})$ for an appropriate value of $C_d$, so we focus on the generic case. For any $\mathbf{v}\in \mathbf{S}^{d-1}$, there is $\mathbf{u}\in A$ such that $\|\mathbf{v}-\mathbf{u}\|\leq \varepsilon_{n}/2$, so by (1) in Lemma~\ref{lem:local_kakeya} $F_n$ contains a unit segment parallel to $\mathbf{v}$. To prove (b), note that every point $\mathbf{x}\in \mathcal{N}_{\delta_n}(F_n)$, is at distance $\leq \delta_n$ from a point $\mathbf{y}\in F^{I_\mathbf{u},\varepsilon_n,\delta_n}$ for some $\mathbf{u}\in A$. But by (2) in Lemma~\ref{lem:local_kakeya}, $\mathbf{y}$ belongs to the $\sqrt{d-1} \varepsilon_n = \delta_{n-1} - \delta_n$-neighbourhood of $I_{\mathbf{u}}$, with $I_{\mathbf{u}} \subset F_{n-1}$. So $\mathbf{x}$ belongs indeed to the $\delta_n+(\delta_{n-1}-\delta_n)=\delta_{n-1}$-neighbourhood of $F_{n-1}$. Finally, by the union bound and (3) in Lemma~\ref{lem:local_kakeya}, we have 
\begin{equation*}
    |\overline{\mathcal{N}_{2\delta_n}(F_n)}|\lesssim_d \#|A| \varepsilon_n^{d-1} f(\delta_n/\varepsilon_n) \lesssim_d f(2\sqrt{d-1} \delta_{n-1}).
\end{equation*}
as required. The last inequality is because $\#|A|\lesssim_d \varepsilon_n^{-(d-1)}$, $f$ is non-decreasing and $\delta_n/\varepsilon_n = \sqrt{d-1} \delta_{n-1}/(1-\delta_{n-1}) \leq 2\sqrt{d-1} \delta_{n-1}$. Hence, we have obtained the required sequence of sets $F_n$, and we now consider 
\begin{equation*}
    E=\bigcap_{n=1}^\infty\overline{\mathcal{N}_{\delta_n}(F_n)}.
\end{equation*}
$E$ is a compact Kakeya set, as it is the limit of a decreasing sequence of Kakeya sets. Furthermore, for any $\delta\in (0,1/2]$, there exists a unique $n\in \N$ so that $\delta_{n+1}<\delta\leq \delta_{n}$. Hence, 
\begin{equation*}
    |\mathcal{N}_\delta(E)|\leq |\overline{\mathcal{N}_{2\delta_{n}}(F_{n})}|\lesssim_df(C_d\sqrt{\delta_{n}})=f(C_d\sqrt[4]{\delta_{n+1}})\lesssim_d \widetilde{f}(\delta),
\end{equation*}
where in the last inequality we have used the non-decreasing property of $f$. This proves the proposition.
\end{proof}

With the above, we immediately obtain a proof of Theorem \ref{thm: main_kakeya}:
\begin{proof}[Proof of Theorem \ref{thm: main_kakeya}]
By Lemma \ref{lemma: lemma kakeya}, we know that for every $\delta\in (0,1)$ there exists a Kakeya set $G^\delta$ such that $|\overline{\mathcal{N}_{\delta}(G^\delta)}|\lesssim_d f(\delta)$, with $f(\delta)=|\log\delta|^{1-d}$. Hence by Proposition \ref{prop: prop kakeya}, there exists a Kakeya set $E$ such that $\forall \delta\in (0,1)$ $|\mathcal{N}_\delta(E)|\lesssim_d f(C_d\sqrt[4]{\delta})\lesssim_d |\log\delta|^{1-d}$, using the properties of the logarithmic function.
\end{proof}

\subsection{On the optimality of Theorem~\ref{thm: main_kakeya}}\label{subsection:optimality}

We now explain why, according to some conjectures in the field, Theorem~\ref{thm: main_kakeya} should be optimal.

 Let $\delta>0$ be a small parameter and let $\Phi:\R\rightarrow\R$ a smooth function of compact support in $[-1,1]$. We define the Fourier multiplier $S^\delta$ on $\R^d$ as
\begin{equation*}
    \widehat{(S^\delta f)}(\xi)=\Phi\Bigg(\frac{|\xi|-1}{\delta}\Bigg)\hat{f}(\xi).
\end{equation*}
We decompose the $\delta$-neighbourhood $\{\xi:||\xi|-1|\leq \delta\}$ of $\mathbf{S}^{d-1}$ into coin-shaped pieces $E_\alpha$ of tangential dimensions $\delta^{1/2}\times\dots\times \delta^{1/2}$ and radial dimension $\delta$, and correspondingly the operator $S^\delta=\sum_\alpha S_\alpha$ where $S_\alpha$ is a Fourier multiplier operator with smooth multiplier $\phi_\alpha$ supported in and adapted to $E_\alpha$. We shall refer to
\begin{equation*}
    ||S^\delta f||_{L^{\frac{2d}{d-1}}(\R^d)}\leq C_d\Bigg|\Bigg|\Bigg(\sum_\alpha|S_\alpha f|^2\Bigg)^{1/2}\Bigg|\Bigg|_{L^{\frac{2d}{d-1}}(\R^d)},
\end{equation*}
where $C_d$ is independent of $\delta$ and $f$, as the reverse Littlewood--Paley conjecture. When $d=2$ it was proved by Fefferman in \cite{0b62136a12e44b6ca6ef02477cb50c4e}, but in higher dimensions it remains open. 

Indeed, Carbery showed \cite{carbery2015remarkreverselittlewoodpaleyrestriction} that the reverse Littlewood-Paley conjecture implies that the Kakeya maximal operator $K_\delta$ has norm $O(|\log \delta|^{\frac{d-1}d})$ from $L_{d}(\R^d) \to L_d(\mathbf{S}^{d-1})$, where 
\begin{equation*}
    K_\delta f(u)=\sup_{T\parallel u} \frac{1}{|T|} \int_T |f|,
\end{equation*}
the supremum being over all $\delta$-tubes in the direction $u$. We warn the reader that in \cite{carbery2015remarkreverselittlewoodpaleyrestriction}, the result is stated in terms of the maximal operator $M_{1/\delta} f(u) =\sup_{T\parallel u} \frac{1}{|T/\delta|} \int_{T/\delta} |f|$ (same supremum), but the proof shows it for $K_\delta$ rather than $M_{1/\delta}$. The $L_d\to L_d$ norms of $K_\delta$ and $M_{1/\delta}$ are related by $\|M_{1/\delta}\| = \delta \|K_\delta\|$. 

If $E$ is a Kakeya set, then its $\delta$-neighbourhood $\mathcal{N}_\delta(E)$ contains a $\delta$-tube in every direction. This means that if $f$ is the indicator of $\mathcal{N}_\delta(E)$, then $K_\delta f=1$, so we get $1 \leq  O( |\log \delta|^{\frac{d-1}d} |\mathcal{N}_\delta(E)|^{\frac 1 d})$, or equivalently $|\mathcal{N}_\delta(E)|\geq C_d |\log \delta|^{-(d-1)}$.

\section{Application to harmonic analysis}\label{sec: harmonic}
\subsection{Besov spaces with logarithmic smoothness}\label{subsection:besov}
Let $(W_n)_{n \geq 0}$ be a Littlewood-Paley partition. This means that there are two functions $w_0,w \in C^\infty_c(\R)$ with $0$ not in the support of $w$ with the following property: if for $n\geq 0$ we set $w_n(x) = w(2^{-n}|x|)$, we have
\[ \sum_{n\geq 0} w_n(x) = 1\]
and $W_n$ is the function in the Schwartz class whose Fourier transform is $w_n$. For $p,q \in [1,\infty]$ and $b,s \in \R$, we then set
\[ B_{p,q}^{s,b}(\R) = \big\{ f \in \mathcal{S}'(\R)\mid  n \mapsto 2^{sn} (1+n)^b \|W_n \ast f\|_{L_p(\R)} \in \ell_q(\N)\big\}.\]
It is a Banach space for the natural norm, and it is a standard fact that the space does not depend on the choice of the Littlewood-Paley partition.
\begin{example}\label{ex:powerLogBesov}
Let $\beta>0$. The function $t \mapsto |\log \frac{(1-e^t)_+}{2}|^{-\beta}$ belongs to $B_{\infty,\infty}^{0,b}(\R)$ if and only if $b \leq \beta$.
\end{example}
\begin{proof}
Let $f_\beta (t)= |\log \frac{(1-e^t)_+}{2}|^{-\beta}$. It is a $C^\infty$ on $(-\infty,0)$, $0$ on $(0,\infty)$, and, as $t\to 0^-$,
\begin{equation}\label{eq:asymptotics_f_beta}
    f_\beta(t) \sim \frac{1}{|\log |t| |^{\beta}} \textrm{ and } f'_\beta(t) \sim -\frac{\beta}{|t|  |\log |t| |^{1+\beta}}.
\end{equation}
Let $\varphi \in C^\infty_c(\R)$ equal to $1$ on a neighbourhood of $0$, and consider $g_\beta = \varphi f_\beta$. Then $f_\beta - g_\beta=(1-\varphi) f_\beta$ is Lipschitz, so $\|W_n \ast (f_\beta-g_\beta)\|_\infty = O(2^{-n})$ as $n\to \infty$. Moreover, it follows by integration by part that
\[ \widehat{g_\beta}(\xi) = \frac{1}{2i\pi\xi}\int_0^{\infty} g_\beta'(-t) e^{2i\pi t\xi} dt.\]
By decomposing the integral as $\int_0^{\varepsilon/|\xi|} + \int_{\varepsilon/|\xi|}^\infty$ for $\varepsilon = (\log |\xi|)^{-1/2}$, we see from \eqref{eq:asymptotics_f_beta} that \[\widehat{g_\beta}(\xi) \sim \frac{1}{2i\pi \xi (\log |\xi|)^\beta}\]
as $|\xi|\to \infty$. In particular, we have
\[ \sup_{\xi \in \mathrm{supp}(\widehat{W_n})} \big|\widehat{g_\beta}(\xi) - \frac{1}{2i\pi n^\beta \xi}\big| = o(n^{-\beta}).\]
But the Fourier transform of $\chi_{(-\infty,0)}$ coincides on $\R\setminus\{0\}$ with the function $\frac{1}{2i\pi \xi}$, so, by the inequality $\|W_n \ast F\|_\infty \leq \|\widehat{W_n}\|_1 \sup_{\xi \in \mathrm{supp}(\widehat{W_n})} |\widehat{F}(\xi)|$, this implies that
\[ \|W_n \ast (g_\beta - n^{-\beta}\chi_{(-\infty,0)})\|_\infty = o(n^{-\beta}),\]
so
\[ \|W_n \ast f_\beta\|_\infty = n^{-\beta} \|W_n \ast \chi_{(-\infty,0)}\|_\infty + o(n^{-\beta}) = n^{-\beta} \|W_1 \ast \chi_{(-\infty,0)}\|_\infty + o(n^{-\beta}) .\]
The last equality is a change of variable. This concludes the proof, because $\|W_1 \ast \chi_{(-\infty,0)}\|_\infty$  is a nonzero real number, and therefore $\sup_{n} (1+n)^{b} \|W_n\ast f_\beta\|_\infty <\infty$ if and only if $\sup_{n} (1+n)^{b-\beta}<\infty$, that is $b \leq \beta$.
\end{proof}

\subsection{Quantitative aspects of Kakeya sets and radial Fourier multipliers}
Recall the notation $f_d(\delta)$ from  subsection~\ref{subsection:tubes}. If $m:\R^d \to \C$ is a bounded measurable function, the Fourier multiplier with symbol $m$ is the map sending a function $f$ to the function whose Fourier transform is $m \hat{f}$. When this map extends to a bounded map on $L_p(\R^d)$, we say that it is $L_p$-bounded. In a recent work \cite{delasalle2026kakeyaconjecturehighranklattice}, the first-named author extended Fefferman's ball multiplier theorem as follows:
\begin{theorem}\label{thm:regularity_Fourier_multipliers}\cite{delasalle2026kakeyaconjecturehighranklattice}
Let $1<p<\infty$ and $m:(0,\infty) \to \R$ a bounded measurable function such that $T$, the Fourier multiplier with symbol $m(|\cdot|)$ is bounded on $L_p(\R^d)$. Then $\varphi:t\in \R \mapsto m(\exp t)$ satisfies
  \[ \|W_n\ast \varphi\|_\infty \leq C_d \|T\|_{L_p \to L_p} \inf_{\delta \in (0,1)} \left(f_d(\delta)^{\big|\frac 1 p - \frac 1 2\big|} + \frac{1}{2^n \delta^2} \right)\]
for all $n \geq 0$.
\end{theorem}

Theorem~\ref{thm:main_tubes} implies the following, which immediately implies Corollary~\ref{cor:radial_mult_sobolev}.
\begin{lemma}\label{lemma:computation_inf}
For every integer $d \geq 2$, there is a constant $C$ such that for every $p \in (1,\infty)$ and every integer $n \geq 0$,
\[ \inf_{0<\delta<1} f_d(\delta)^{\big|\frac 1 p - \frac 1 2\big|}  + \frac{1}{2^n \delta^2} \leq \frac{C}{(1+n)^{(d-1) \big|\frac 1 p - \frac 1 2\big|}}.\]
\end{lemma}
\begin{proof}
Let $\alpha = (d-1) \big|\frac 1 p - \frac 1 2\big|$. Theorem~\ref{thm:main_tubes} asserts that $f_d(\delta)^{\big|\frac 1 p - \frac 1 2\big|} \lesssim_d |\log \delta|^{-\alpha}$. For $\delta = \min(1/2,2^{-n/2} (1+n)^{\alpha/2})$, both terms $f_d(\delta)^{\big|\frac 1 p - \frac 1 2\big|}$ and $\frac{1}{2^n \delta^2}$ are $O_d( (1+n)^{-\alpha})$.
\end{proof}
Adapting the proof of Theorem~\ref{thm:regularity_Fourier_multipliers} to spherical harmonic analysis and Schur multipliers, the second-named author also proved the following. Here is $m$ is a bounded measurable map $X\times X\to \C$, the Schur multiplier with symbol $m$ is the map sending the Hilbert-Schmidt operator on $L_2(X)$ with kernel $(A_{x,y})_{x,y \in X}$ to the one with kernel $(m(x,y) A_{x,y})_{x,y \in X}$. When this map extends to a bounded linear map on $S_p(L_2(X))$, we say that it is $S_p$-bounded.
\begin{theorem}\cite{delasalle2026kakeyaconjecturehighranklattice}
Let $1<p<\infty$ and $m:(-1,1) \to \R$ be a bounded measurable function such that $S$, the Schur multiplier with symbol $(x,y)\mapsto m(\langle x,y\rangle)$ is bounded on $S_p(L_2(\mathbb{S}^d))$. Then $\varphi:t\in \R \mapsto m( \cos t)$ satisfies
  \[ |W_n\ast \varphi(\theta)| \leq C_d \|S\|_{S_p \to S_p} \inf_{\delta \in (0,1)} \left(f_d(\delta)^{\big|\frac 1 p - \frac 1 2\big|} + \frac{1}{|2^n \delta^2 \sin \theta|^{\frac 1 3}} \right)\]
for all $n \geq 0$.
\end{theorem}
With the same proof as in Lemma~\ref{lemma:computation_inf}, we obtain the concrete bound as a consequence of Theorem~\ref{thm:main_tubes}
  \[ |W_n\ast \varphi(\theta)| \leq C_d \|S\|_{S_p \to S_p} \min(1, (1+\log(2^n |\sin \theta|)) ^{-(d-1)|\frac 1 p - \frac 1 2|}.\]
In particular, if $m$ is supported in $[-c,c]$ for some $c<1$, we obtain that $\varphi \in B_{\infty,\infty}^{0,b}(\R)$ for $b=(d-1)|\frac 1 p - \frac 1 2|$.

\section{Interactive plots}
We have coded some of the constructions described in this paper, for the cases $d=2$ and $d=3$. Once compiled, the simulations allow one to rotate, zoom and toggle sliding bars that control the translation parameters. We include in separate files the python code to visualise the Perron tree construction associated to the fair subdivision of the $2$-simplex and to the 2-dimensional unit cube. The reader can find the \textit{Jupyter Notebooks} by clicking on the following GitHub repository:
\newline \url{https://github.com/rafaeljfernandezd/A-CONSTRUCTION-OF-KAKEYA-SETS-IN-ARBITRARY-DIMENSION0}

\begin{appendix}\label{appendix}
\section{Subdivisions of a  polytope}
We will now prove that we can obtain a Perron tree out of pyramids whose bases are any given  polytope, as long as we start with a sufficiently fine partition of it.

\begin{lemma}\label{lem:NSC_for_shrinking}
Let $(\Sigma_i)_{i \in I}$ be a partition of a polytope $\Sigma\subset \R^d$. Then \eqref{eq:assumption_shrinking} holds for some $c\in (0,1)$ if and only if for every $i$, the intersection of the faces of $\Sigma$ that meet $\Sigma_i$ is non-empty.
\end{lemma}
Observe that when $\Sigma$ is a simplex, the condition in the lemma is equivalent to saying that no $\Sigma_i$ meets all the co-dimension $1$ faces of $\Sigma$. Lemma~\ref{lem:NSC_for_shrinking} follows immediately from the following:
\begin{lemma}
Let $\Sigma \subset \R^d$ be a polytope, $\Lambda \subset \Sigma$ be closed and $\mathbf{x} \in \Sigma$. The following are equivalent:
\begin{itemize}
    \item There is $c \in (0,1)$ such that $\Lambda \subset \mathcal{H}_{\mathbf{x}}^c(\Sigma)$,
    \item $\mathbf x$ belongs to all the faces of $\Sigma$ that meet $\Lambda$.
\end{itemize}
\end{lemma}
\begin{proof} Write $\Sigma = \bigcap_{k=1}^K \{\varphi_k \leq 0\}$ as a finite intersection of half-spaces, where $\varphi_k:\R^d \to \R$ are affine maps. Then we have $\mathcal{H}_{\mathbf{x}}^c(\Sigma) = \bigcap_{k=1}^K \{\varphi_k \leq (1-c) \varphi_k(\mathbf{x})\}$, so $\Lambda \subset \mathcal{H}_{\mathbf{x}}^c(\Sigma)$ is equivalent to $\max_\Lambda \varphi_k \leq (1-c) \varphi_k(\mathbf{x})$ for every $k$. If $\Lambda$ meets the face $\{\varphi_k=0\}$, this inequality holds if and only if $0 =\varphi_k(\mathbf{x})$, that is $\mathbf{x}$ belongs to it. If $\Lambda$ does not meet the face $\{\varphi_k=0\}$, then $\max_\Lambda \varphi_k<0$ because $\Lambda$ is closed, and this inequality holds for $c$ close enough to $1$.
\end{proof}
\begin{remark}\label{rem: choice point}
From the above, we see that each $\mathbf{x}_i$ in condition \eqref{eq:assumption_shrinking} must be chosen in the intersection of all the faces of $\Sigma$ that meet $\Sigma_i$, and this set must be non-empty. If no face of $\Sigma$ intersects $\Sigma_i$, then $\mathbf{x}_i\in \Sigma$ can be chosen arbitrarily. 
\end{remark}

As a consequence of the previous lemmas, we deduce the following:
\begin{proposition}\label{prop: fine polytope}
If $(\Sigma_i)_{i \in I}$ is a fair partition of a  polytope with $|I|>1$, then there is $k\in \N$ such that the repeated partition $(\Sigma_i^{(k)})_{i \in I^k}$ satisfies \eqref{eq:assumption_shrinking} for some $c=c_{\Sigma,(\Sigma_i)_{i\in I}}\in (0,1)$.
\end{proposition}
\begin{proof}
Let $\mathcal{F}$ be the set of faces of $\Sigma$. For any family of faces $S\subset \mathcal{F}$ that has empty intersection, consider the number 
\begin{equation*}
    \delta(S)=\min_{\mathbf{x}\in \Sigma}\max_{F\in S}\{d(\mathbf{x},F)\},
\end{equation*}
which is positive because the faces in $S$ have empty intersection. Define
\begin{equation*}
    \delta=\min_{\substack{
S\subset \mathcal{F}\\
\bigcap_{F\in S}F=\emptyset}}\{\delta(S)\},
\end{equation*}
which also positive due to the finitude of the families of faces with empty intersection. Now choose $k\in \N$ sufficiently large so that after carrying out $k$ successive subdivisions of our initial partition, we obtain a fair partition $(\Sigma_i^{(k)})_{i\in I^k}$ with  $\operatorname{diam}(\Sigma_i^{(k)})=|I|^{-k/(d-1)}\operatorname{diam}(\Sigma)<\delta$. For such choice, consider $S$ the set of faces of $\Sigma$ that meet $\Sigma_i^{(k)}$. We have $\delta(S) \leq \operatorname{diam}(\Sigma_i^{(k)})<\delta$, so $\bigcap_{F\in S} F\neq \emptyset$. It follows from Lemma~\ref{lem:NSC_for_shrinking} that $(\Sigma_i^{(k)})_{i \in I^k}$ satisfies \eqref{eq:assumption_shrinking}.
\end{proof}

\begin{figure}[htbp]
  \centering
  \begin{minipage}{0.46\textwidth}
    \includegraphics[width=\textwidth]{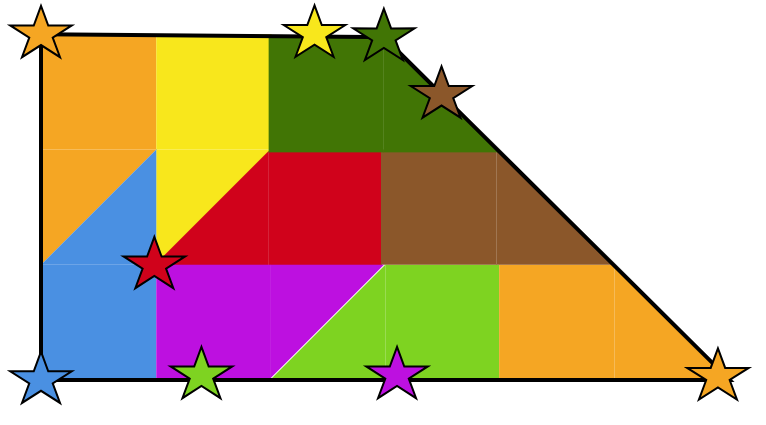}
  \end{minipage}
  \hfill
  \begin{minipage}{0.44\textwidth}
    \includegraphics[width=\textwidth]{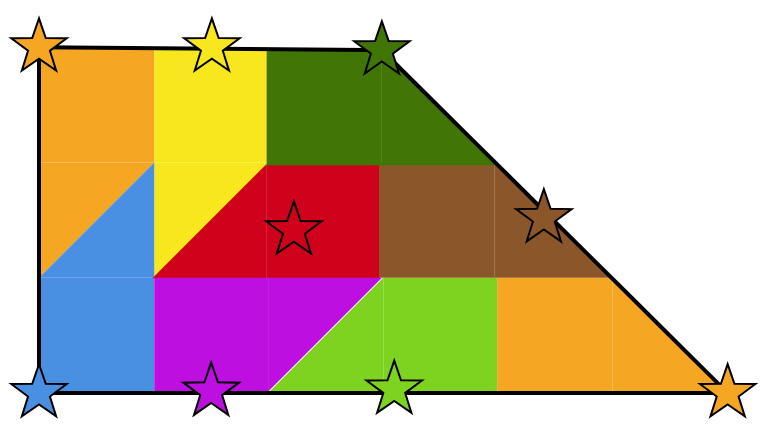}
  \end{minipage}
  \caption{Fair polytopal subdivision of a trapezium into 9 cells. Each star, corresponding to its same-coloured cell, represents an arbitrary choice of $(\mathbf{x}_i)_{i\in [9]}$ that is consistent with the rules of Remark \ref{rem: choice point}. The first picture's choice makes it impossible for $(\Sigma_i,\mathbf{x}_i)$ to be affinely separable, whilst the second choice makes it possible since $\mathbf{x}_i\in \Sigma_i$.}
  \label{fig: trapezium}
\end{figure}

\begin{figure}[htbp]
  \centering
    \includegraphics[width=0.35\textwidth]{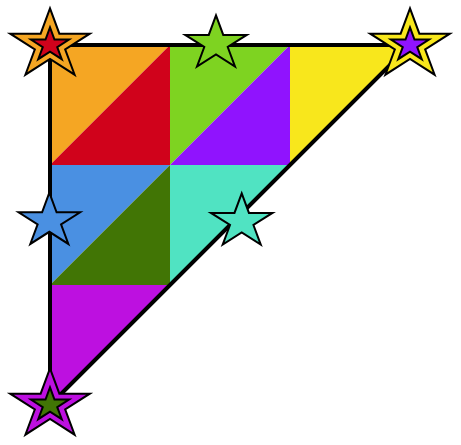}
  \hfill
  \caption{Fair simplicial subdivision of a triangle into 9 cells. Each star, corresponding to its same-coloured cell, represents an arbitrary choice of $(\mathbf{x}_i)_{i\in [9]}$ that is consistent with the rules of Remark \ref{rem: choice point}. The only allowed $\mathbf{x}_i$'s for the red, dark-green and indigo cells are not in $\Sigma_i$. However, this choice still makes $(\Sigma_i,\mathbf{x}_i)$ affinely separable.}
  \label{fig: trapezium}
\end{figure}
\begin{remark}
Note that we have not proven that any sufficiently fine polytopal partition has a family $(\Sigma_i,\mathbf{x}_i)_{i\in I}$ which is affinely separable. 
\end{remark}

\section{A fair subdivision of the $d$-simplex}\label{ap: simplex}
In this section, we first describe the Coxeter-Freudenthal-Kuhn simplicial decomposition of a $(d-1)$-simplex, and then we study its geometrical properties. The statement of Theorem~\ref{thm:CFK} was apparently first proved for powers of $2$ by Freudenthal \cite{Freudenthal42}, answering a question of Brouwer. But, up to affine transformations, Freudenthal's construction can also be derived from Coxeter's $\tilde{A}_{d-1}$ tessellation \cite{Coxeter34} (see \cite{boissonnat}). There are several expositions of the Coxeter-Freudenthal-Kuhn for arbitrary $q$. We will follow the one in the work of Mirzakhani-Vondrák \cite{RefWorks:mirzakhani2015sperners} where we learnt about it. Other presentations include \cite{RefWorks:edelsbrunner1999edgewise}, see also \cite{boissonnat} for a historical account.
 
\subsection{Fair simplicial subdivision of the $d$-simplex}\label{subsection: regular subdivision}
 Let $q\geq 1$ and consider the $(d-1)$-dimensional simplex
\begin{equation*}
    R_{d,q}=\{\mathbf{y}\in \R_+^{d-1}\mid 0\leq y_1\leq y_2\leq \dots\leq y_{d-1}\leq q\}.
\end{equation*}
We now describe the subdivision of $R_{d,q}$ into $q^{d-1}$ pieces constructed in \cite{RefWorks:mirzakhani2015sperners}. Such subdivision will consist of cells whose vertices are the set
\begin{equation*}
    W_{d,q}=\{\mathbf{v}\in \Z^{d-1}_+\mid 0\leq v_1\leq v_2\leq \dots\leq v_{d-1}\leq q\}.
\end{equation*}

A cell of the subdivision of $R_{d,q}$ is indexed by a vertex $\mathbf{w}\in W_{d,q-1}$ and a permutation $\pi: [d-1]\to [d-1]$. The permutation $\pi$ should be \textit{consistent} with $\mathbf{w}$ in the sense that whenever $w_i=w_{i+1}$, we have $\pi(i)<\pi(i+1)$. For any such pair $(\mathbf{w},\pi)$, the respective cell is defined as
\begin{equation*}
    \sigma(\mathbf{w},\pi)=\{\mathbf{y}\in \R^{d-1}_+\mid 0\leq (\mathbf{y}-\mathbf{w})_{\pi^{-1}(1)}\leq (\mathbf{y}-\mathbf{w})_{\pi^{-1}(2)}\leq \dots\leq (\mathbf{y}-\mathbf{w})_{\pi^{-1}(d-1)}\leq 1\}.
\end{equation*}
We warn the reader that this formula is correct, but it does not match the one in \cite{RefWorks:mirzakhani2015sperners} where the inverse was missing in $\pi^{-1}$; with the wrong formula it is not true that $\sigma(\mathbf{w},\pi)\subset R_{d,d}$.

Mirzakhani and Vondrák \cite{RefWorks:mirzakhani2015sperners} prove that the collection \begin{equation}\label{eq:MVsubdivision}\big\{ \sigma(\mathbf{w},\pi) \mid \mathbf{w}\in W_{d,q-1}, \pi \in \mathrm{Sym}(d-1)\textrm{ which is consistent with }\mathbf{w}\big\}
\end{equation}
is a simplicial subdivision into $q^{d-1}$ pieces. In fact, it is a $\mathrm{Sym}(d-1)$-fair subdivision for the natural $\mathrm{Sym}(d-1) \to \mathrm{O}(d-1)$:
\[ \sigma(\mathbf{w},\pi) = \frac 1 q \pi(R_{d,q})+\mathbf{w}.\]

\subsection{Corners and subdivision properties}
Fix $q\geq d$. The  polytope $\Sigma=R_{d,q}$ is a $(d-1)$-simplex in $\R^d$ with extreme points $\mathbf{x}_1=(q,q,\dots,q)$, $\mathbf{x}_2=(0,q,\dots,q)$, ..., $\mathbf{x}_{d-1}=(0,0,\dots,0,q)$, $\mathbf{x}_{d}=(0,0,\dots,0)$. 
Consider the ``corner pieces''
\begin{equation*}
    \widetilde{\Sigma_k}=\mathcal{H}_{\mathbf{x}_i}^{\frac{q-1}{q}}(\Sigma),\qquad k\in [d].
\end{equation*}
The next proposition asserts that, for these choices of parameters, the subdivision we just introduced has all the desired properties to perform our Perron tree construction with it:

\begin{figure}[htbp]
  \centering
  \makebox[0pt][c]{\hspace{9cm}\begin{minipage}{0.3\textwidth}
    \includegraphics[width=\textwidth]{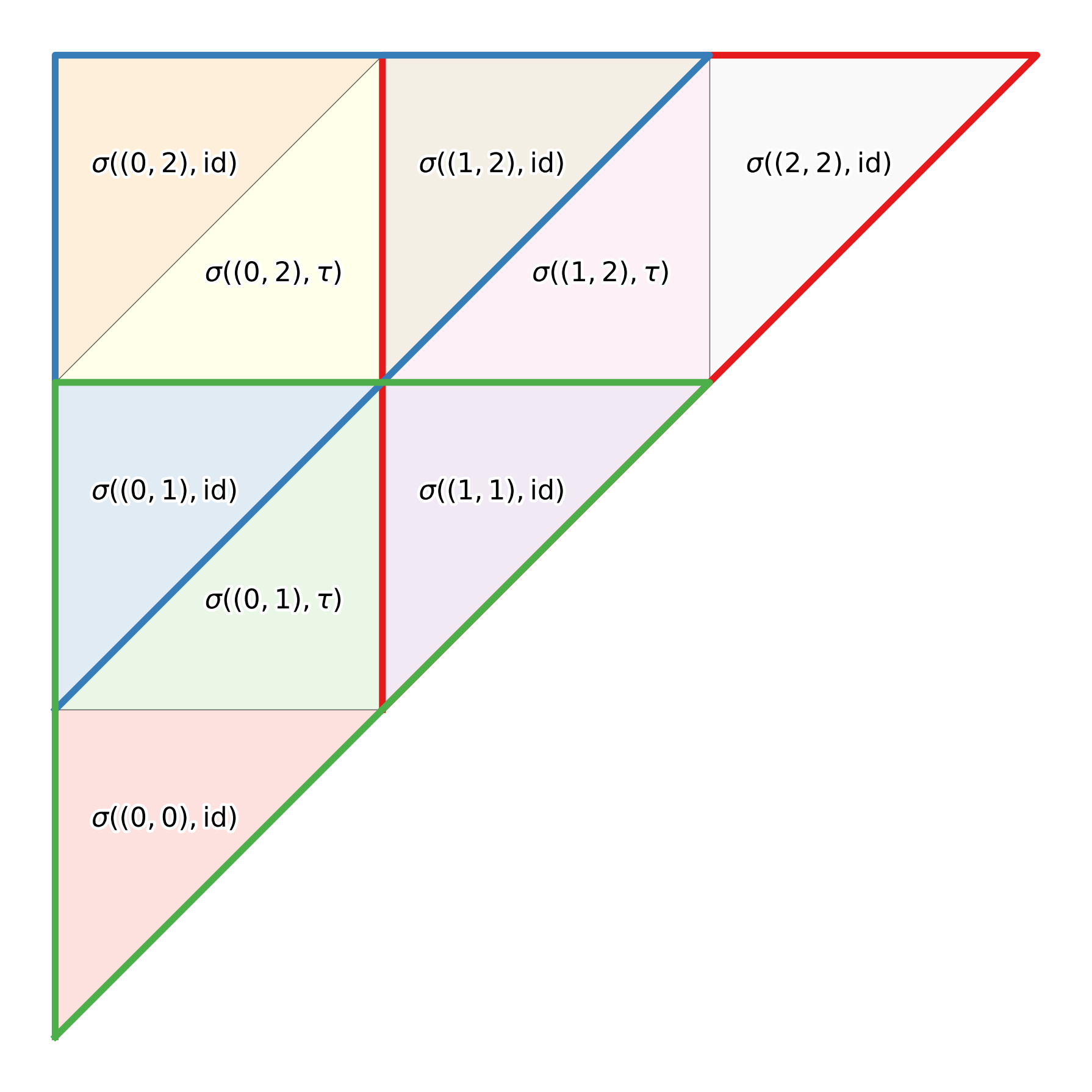}
  \end{minipage}}
  \hfill
 \makebox[0pt][c]{\hspace{-9cm} \begin{minipage}{0.4\textwidth}
    \includegraphics[width=\textwidth]{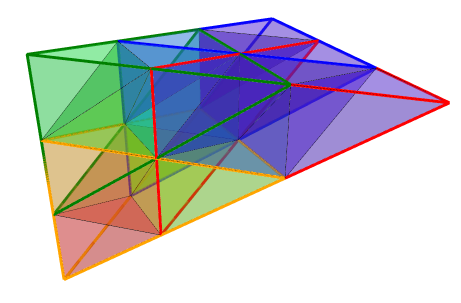}
  \end{minipage}}
  \caption{Plot of the fair simplicial subdivisions of $R_{3,3}$ and $R_{4,3}$ described in this section, together with the wire frames for their respective corner pieces. The labelling of each cell has been made explicit in $R_{3,3}$, with $\tau=(1,2)$.}
  \label{fig: MV subdiv}
\end{figure}
\begin{proposition}\label{lemma: cells embedding} Let $(\Sigma_i)_{i\in I}=(\sigma(\mathbf{w},\pi))_{(\mathbf{w},\pi)}$ be the Coxeter-Freudenthal-Kuhn fair subdivision  of $\Sigma=R_{d,q}$ into $q^{d-1}$ pieces with $q \geq d$, then \eqref{eq:assumption_shrinking} holds with $c=\frac{q-1}{q}$. It is also affinely separated.
\end{proposition}
\begin{proof}
Let us define dummy variables $y_0=w_0=0$ and $y_d=w_d=q$ and the difference functionals $\Delta_k(\mathbf{y})=y_k-y_{k-1}$ for $k\in [d]$. With these definitions, it is immediate to check that 
\begin{equation*}
    \Sigma=\{\mathbf{y}\in \R^{d-1}\mid  \Delta_k(\mathbf{y})\geq 0,\;k\in [d]\},\qquad \widetilde{\Sigma_k}=\{\mathbf{y}\in \Sigma\mid \Delta_k(\mathbf{y})\geq 1\}.
\end{equation*}
A cell $\sigma(\mathbf{w},\pi)$ is formed by points $\mathbf{y}=\mathbf{w}+\mathbf{z}$, where $z_{\pi^{-1}(1)}\leq \dots z_{\pi^{-1}(d-1)}\leq 1$. This imposes $z_j\in [0,1]$ for all $j\in [d-1]$, and fixes the relative ordering of the coordinates:
\begin{equation*}
    z_{k-1}\leq z_k\quad \Longleftrightarrow\quad \pi(k-1)<\pi(k).
\end{equation*}
For any $\mathbf{y}\in \sigma(\mathbf{w},\pi)$, we substitute $y_j=w_j+z_j$:
\begin{equation*}
    \Delta_k(\mathbf{y})=(w_k-w_{k-1})+(z_k-z_{k-1}).
\end{equation*}
Let $D_k=w_k-w_{k-1}$. Since $\mathbf{w}$ is non-decreasing, $D_k\geq 0$ is an integer. Because $z_j\in [0,1]$, then $(z_k-z_{k-1})\in [-1,1]$. We analyse the range of $\Delta_k(\mathbf{y})$ over the whole cell:
\begin{enumerate}
    \item If $D_k\geq 2$: $\Delta_k(\mathbf{y})\geq 2-1=1$. 
    \item If $D_k=0$: $\Delta_k(\mathbf{y})\leq 0+1=1$.
    \item If $D_k=1$: \begin{itemize}
        \item If $\pi(k-1)<\pi(k)$, then $z_{k-1}\leq z_k$, and thus $\Delta_k(\mathbf{y})\geq 1+0=1$.
        \item If $\pi(k-1)>\pi(k)$, then $z_{k-1}\geq z_k$, and thus $\Delta_k(\mathbf{y})\leq  1+0=1$ 
    \end{itemize}
\end{enumerate}
In every possible configuration, the cell $\sigma(\mathbf{w},\pi)$ lies completely on one side of the hyperplane $\Delta_k=1$, i.e. either $\Delta_k(\sigma)\leq 1$ or $\Delta_k(\sigma)\geq 1$.  Therefore, no cell straddles the boundary of any corner piece.

Moreover, there is at least one $k$ such that $\sigma \subset \widetilde{\Sigma}_k$. Indeed, otherwise by the preceding if $\mathbf{y}$ belongs to the interior of $\sigma$ then $\Delta_k(\sigma)<1$ for all $k\in [d]$, and
\[ q= y_d-y_0 = \sum_{k=1}^d \Delta_k(\mathbf y) <d,\]
a contradiction with the assumption that $q\geq d$. 

Define $\mathbf{x}_\sigma$ as $\mathbf{x}_k$ for the smallest $k$ such that $\Delta_k \geq 1$ on $\sigma$. This means that $\sigma$ is contained in $\widetilde{\Sigma}_k$, so \eqref{eq:assumption_shrinking} holds with $c=\frac{q-1}{q}$.

Let us show that partition $(\sigma,\mathbf{x}_\sigma)_\sigma$ is affinely separated. Let $\sigma,\sigma'$ be two distinct cells, and $k,k'$ the indices such that $\mathbf{x}_\sigma= \mathbf{x}_k$ and $\mathbf{x}_{\sigma'}=\mathbf{x}_{k'}$. If $k=k'$ any affine map separating $\sigma$ from $\sigma'$ also separates $\mathbf{x}_{\sigma}$ from $\mathbf{x}_{\sigma'}$. Otherwise we can assume $k<k'$. 
Then $\Delta_k \geq 1$ on $\sigma$, while $\Delta_k \leq 1$ on $\sigma'$. Moreover, we have $\Delta_k(\mathbf{x}_k)=q > 0=\Delta_k(\mathbf{x}_{k'})$, so the affine map $\Delta_k-1$ separates $(\sigma,\mathbf{x}_\sigma)$ from $(\sigma',\mathbf{x}_{\sigma'})$. 
\end{proof}

\begin{figure}[htbp]
  \centering
 \makebox[0pt][c]{\hspace{9cm} \begin{minipage}{0.3\textwidth}
    \includegraphics[width=\textwidth]{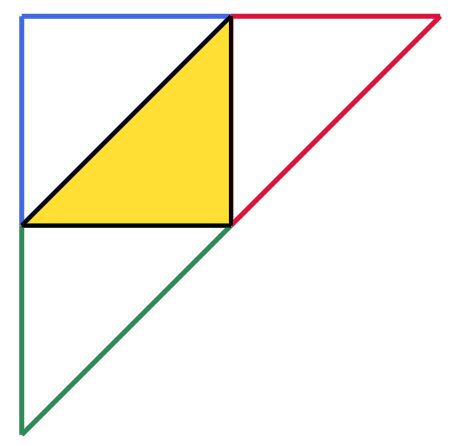}
  \end{minipage}}
  \hfill
  \makebox[0pt][c]{\hspace{-9cm}\begin{minipage}{0.4\textwidth}
    \includegraphics[width=\textwidth]{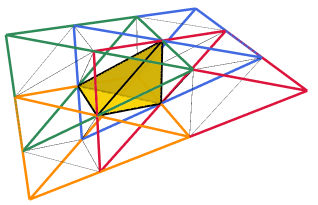}
  \end{minipage}}
  \caption{Plot of the fair simplicial subdivisions of $R_{3,2}$ and $R_{4,3}$ described in this section, together with the wire frames for their respective corner pieces for $c=\frac{q-1}{q}$. The highlighted cell is not contained in any corner piece, which illustrates the necessity of the condition $q\geq d$ for Proposition \ref{lemma: cells embedding} to hold true.}
\end{figure}

\end{appendix}
\bibliographystyle{alpha}
\bibliography{biblio}

\end{document}